\documentclass{amsart}
\usepackage {amssymb}
\usepackage {amsmath}
\usepackage{amsthm}
\usepackage{graphicx}
\usepackage {amscd}
\usepackage {epic}
\usepackage[alphabetic]{amsrefs}
\usepackage[colorlinks, citecolor=blue]{hyperref}
\usepackage{bbm}
\usepackage{enumitem}
\usepackage{mathrsfs}
\usepackage[all,color]{xy}
    \newenvironment{dedication}
        {\vspace{6ex}\begin{quotation}\begin{center}\begin{em}}
        {\par\end{em}\end{center}\end{quotation}}

\newcommand{\dlt}[0]{\operatorname{dlt}}

\newcommand{\sddb}{{\sqrt{-1}\partial\bar{\partial}}}
\newcommand{\vphi}{{\varphi}}

\newcommand{\cD}{\mathcal{D}}
\newcommand{\cY}{\mathcal{Y}}
\newcommand{\cE}{\mathcal{E}}
\newcommand{\la}{\langle}
\newcommand{\ra}{\rangle}
\newcommand{\cR}{\mathcal{R}}
\newcommand{\cL}{\mathcal{L}}
\newcommand{\NA}{{\rm NA}}
\newcommand{\bP}{\mathbb{P}}
\newcommand{\fD}{\mathfrak{D}}
\newcommand{\ka}{\mathfrak{a}}
\newcommand{\bS}{\mathbb{S}}

\newcommand{\Fut}{{\rm Fut}}

\newcommand{\bZ}{{\mathbb{Z} }}
\newcommand{\bQ}{{\mathbb{Q} }}
\newcommand{\bC}{{\mathbb{C} }}

\newcommand{\cX}{{\mathcal{X}}}
\newcommand{\cF}{{\mathcal{F}}}

\newcommand{\DR}{{\mathcal{DR}}}
\newcommand{\ft}{{\mathfrak{t}}}
\newcommand{\D}{{\mathcal{D}}}
\newcommand{\bV}{{\mathbb{V}}}

\newcommand{\mld}{{\rm mld}}
\newcommand{\mult}{{\rm mult}}
\newcommand{\Supp}{{\rm Supp}}

\newcommand{\ord}{{\rm ord}}

\newcommand{\lct}{{\rm lct}}
\newcommand{\vol}{{\rm vol}}
\newcommand{\Val}{{\rm Val}}
\newcommand{\fa}{{\mathfrak{a}}}
\newcommand{\hvol}{\widehat{\rm vol}}

\newcommand{\bk}{\mathbbm{k}}
\newcommand{\bR}{\mathbb{R}}
\newcommand{\cO}{\mathcal{O}}
\newcommand{\gr}{{\rm gr}}
\newcommand{\wt}{{\rm wt}}

\newcommand{\fb}{\mathfrak{b}}
\newcommand{\Cosupp}{\mathrm{Cosupp}}
\newcommand{\ratrank}{\mathrm{rat.rank}\,}
\newcommand{\transdeg}{\mathrm{trans.deg}\,}
\newcommand{\fm}{\mathfrak{m}}
\newcommand{\QM}{\mathrm{QM}}

\newcommand{\bN}{\mathbb{N}}
\newcommand{\ind}{\mathrm{ind}}

\newcommand{\Fvol}{{\rm Fvol}}
\newcommand{\fpt}{\mathrm{fpt}}
\newcommand{\codim}{\mathrm{codim}}

\newtheorem{thm}{Theorem}[section]

\newtheorem{lem}[thm]{Lemma}
\newtheorem{cor}[thm]{Corollary}

\newtheorem{prop}[thm]{Proposition}

\newtheorem{conj}[thm]{Conjecture}

\theoremstyle{definition}
\newtheorem{defn}[thm]{Definition}

\newtheorem{ques}[thm]{Question}    
\newtheorem{que}[thm]{Question}    

\newtheorem{rem}[thm]{Remark}

\newtheorem{defn-thm}[thm]{Definition--Theorem}  
\newtheorem{defn-prop}[thm]{Definition--Proposition}  
\newtheorem{defn-lem}[thm]{Definition--Lemma}  

\theoremstyle{remark}

\input{diagrams.sty}

\begin{document}
\title{A guided tour to normalized volume }

\begin{dedication}
\hspace{0cm}
\vspace*{1cm}{Dedicated to Gang Tian's Sixtieth Birthday with admiration}
\end{dedication}

\author {Chi Li}
\address {Department of Mathematics, Purdue University, West Lafayette, IN 47907-2067, USA}
\email {li2285@purdue.edu}

\author {Yuchen Liu}
\address{Department of Mathematics, Yale University, New Haven, CT 06511, USA}
\email{yuchen.liu@yale.edu}

\author {Chenyang Xu}

\address   {Beijing International Center for Mathematical Research,
       Beijing 100871, China}
\email     {cyxu@math.pku.edu.cn}
\begin{abstract}
This is a survey on the recent theory on minimizing the normalized volume function attached to any klt singularities.
\end{abstract}

\maketitle
\tableofcontents
\section{Introduction}

The development of algebraic geometry and complex geometry has interwoven in the history. One recent example is the interaction between the theory of higher dimensional geometry centered around the minimal model program (MMP), and the existence of `good' metrics on algebraic varieties. Both subjects have major steps forward, whose influences are beyond the subjects themselves, spurring out new progress in topics once people could not imagine. In this note, we will discuss a `local stability theory' of singularities, which in our opinion provides an excellent example on the philosophy that there are many unexpected connections underlying these two different topics.  

Ever since the starting of the theory of MMP in higher dimensions (that is, the dimension is at least three), people understand that a feature of such a theory is that we need to deal with singular varieties. Then it becomes very nature  to investigate this class of singularities for people working on the MMP. To deal with singular varieties in complex geometry is a more recent trend, and it significantly improves people's knowledge on the existence of interesting metrics, even in situations which people originally only want to study smooth varieties. 

It becomes clear now, Kawamata log terminal (klt) singularities form an exceptionally important class of singularities for many reasons: it is the natural class of singular varieties for people to inductively prove deep results in the MMP; it is the class of singularities appearing on degenerations in many natural settings and it carries properties which globally Fano varieties have. 

What we want to survey here is a rather new theory on klt singularities. The picture consists of two closely related parts: firstly, we want to establish a structure which provides a canonically determined  degeneration  to a stable log Fano cone from each  klt singularity; secondly, to construct the degeneration, we need a valuation which minimizes the normalized volume function on the `non-archimedean link', and since such minimum is a deep invariant defined for all klt singularities, we want to explore more properties of this invariant, including calculating it in many cases. 


\subsection{ History} The first prototype of the local stability theory underlies in \cite{MSY06, MSY08}. They find that the existence of  Ricci-flat cone metric on an affine variety with a good action by a torus group $T$ is closely related to the normalized volume minimizing problem. In our language, they concentrate on the valuations induced by the vectors in the Reeb cone provided by the torus action. Later a systematic study of K-stability in the setting of Sasaki geometry is further explored in \cite{CS18,CS15}. 

Consider klt singularities which appear on the Gromov-Hausdorff (GH) limit of K\"aher-Einstein Fano manifolds. At the first sight, we do not know more algebraic structure for these singularities. Nevertheless, by looking at the metric tangent cone, it is shown in \cite{DS17},  built on the earlier works in \cite{CCT02, DS14, Tia13}, that the metric tangent cone of such singularities is an affine $T$-variety with a Ricci-flat cone metric. Furthermore, \cite{DS17} gives a  a two-step degeneration description of the metric tangent cone. They further conjecture that this two-step degeneration should only depend on the algebraic structure of the singularity, but not the metric. 

Then in \cite{Li15}, the  normalized volume function on the `non-archimedean link' of a given klt singularity is defined, and a series of conjectures on normalized volume function are proposed. This  attempt is not only to algebrize the work in \cite{DS17} without invoking the metric, but it is also of a completely local nature. Since then, the investigation on this local stability theory points to different directions. 

In \cite{Blu18}, the existence of a minimum (opposed to only infimum) which was conjectured in \cite{Li15} is affirmatively answered. The proof uses the properness estimates in \cite{Li15} and the observation in \cite{Liu16} that the minimizer can be computed by the minimal normalized multiplicities, and then  skillfully uses the techniques from the study of asymptotic invariants (see \cite{Laz04}). Later in \cite{BL18}, lower semicontinuity of the volume of singularities are also established using this circle of ideas.

In \cite{Li17,LL16}, the case of a cone singularity over a Fano variety is intensively studied, and it was found if we translate the minimizing question for the canonical valuation into a question on the base Fano varieties, what appears is the sign of the $\beta$-invariant developed in \cite{Fuj18,Fuj16, Li17}. 

Built on the previous study of cone singularities, implementing the ideas circled around the MMP in  birational geometry, an effective process of degenerating a general singularity to a cone singularity is established in \cite{LX16}, provided the minimizer is a divisorial valuation. 
In \cite{LX17}, a couple of conjectural properties are added to complete the picture proposed in \cite{Li15}, and now the package is called `{\it stable degeneration conjecture}', see Conjecture \ref{conj-local}. The investigation in \cite{LX16} is also extended in \cite{LX17} to the case when the minimizer is a quasi-monomial valuation with a possibly higher rational rank, where the study involves a considerable amount of new techniques. As a corollary, the first part of Donaldson-Sun's conjecture in \cite{DS17} is answered affirmatively in \cite{LX17}. Later the work is extended in \cite{LWX18} and a complete solution of Donaldson-Sun's conjecture is found. 

Applications to global questions, especially the existence of KE metrics on Fano varieties, are also explored. In \cite{Liu16}, built on the work of \cite{Fuj18}, an inequality to connect the local volume and the global one is proved. Then in \cite{SS17, LiuX17}, via the approach of the `comparison of moduli', complete moduli spaces parametrizing  explicit Fano varieties with a KE metric  are established by studying the local constraint posted by the lower bound of the local volumes.

\subsection {Outline} In the note, we will survey a large part of the results mentioned above. From the perspective of techniques, there are three closely related ways to think about the volume of a singularity: the infima of the normalized volume of valuations, of the normalized multiplicity of primary ideals  or of  the volume of models. The viewpoint using valuations gives the most canonical picture, e.g. the stable degeneration conjecture, but there are less techniques available to directly study the space. The viewpoint using ideals is flexible for many purposes, e.g. taking degenerations. Moreover, though usually working on a single ideal does not give too much advantage over others, working on a graded sequence of ideals really enables one to use the powerful theory on asymptotical invariants for such setting. The third viewpoint of using models allows us to apply the machinery from the MMP theory, and it is the key to degenerate the underlying singularities into cone singularities.  
The interplay among these three circle of techniques is fruitful, and we expect further insight can be made in the future.

In Section \ref{s-def}, we give the definition of the function of the normalized volumes and sketch the basic properties of its minimizer, including the existence. In Section \ref{s-sasaki}, we discuss the theory on searching for Sasaki-Einstein metrics on a Fano cone singularities. The algebraic side, namely the K-stability notions on a Fano cone plays an important role as we try to degenerate any klt singularity to a K-semistable Fano cone. Such an attempt is formulated in the stable degeneration conjecture, which is the focus of Section \ref{s-SDC}. In Section \ref{s-application}, we present some applications, including the torus equivariant K-stability (Section \ref{ss-equiva}), a solution of Donaldson-Sun's conjecture (Section \ref{ss-dsc}) and the K-stability of cubic threefolds (Section \ref{ss-cubic3}). In the last Section \ref{s-ques}, we discuss many unsolved questions, which we hope will lead to some future research. Some of them give new approaches to attack existing problems. 

\bigskip

\noindent {\bf Acknowledgement:} We would like to use this chance to express our deep gratitude to Gang Tian, from whom we all learn a large amount of knowledge related to K-stability questions on Fano varieties in these years. We want to thank Harold Blum, S\'ebastien Boucksom for helpful discussions. 

CL is partially supported by NSF (Grant No. DMS-1405936) and an Alfred P. Sloan research fellowship. 
CX is partially supported by the National Science Fund for Distinguished Young Scholars (11425101).  A large part of the work is written while CX visits Institut Henri Poincar\'e under the program `Poincar\'e Chair'.  He wants to thank the wonderful environment. 

\section{Definitions and first properties}\label{s-def}

\subsection{Definitions}\label{ss-firstdefinition}
In this section, we give the definition of the normalized volume $\hvol_{(X,D),x}(v)$ (or abbreviated as $\hvol(v)$ if there is no confusion) for a valuation $v$ centered on a klt singularity $x\in (X,D)$ as in \cite{Li15}. It consists of two parts: the volume $\vol(v)$ (see Definition \ref{d-vol}) and the log discrepancy $A_{X,D}(v)$ (see Definition \ref{d-logd}).
\bigskip

Let $X$ be a reduced, irreducible variety defined over $\bC$. A \emph{real valuation} of its function field
$K(X)$ is a non-constant map $v\colon K(X)^{\times}\to \bR$, satisfying:
\begin{itemize}
 \item $v(fg)=v(f)+v(g)$;
 \item $v(f+g)\geq \min\{v(f),v(g)\}$;
 \item $v(\bC^*)=0$.
\end{itemize}
We set $v(0)=+\infty$. A valuation $v$ gives rise
to a valuation ring $\cO_v:=\{f\in K(X)\mid v(f)\geq 0\}$.
We say a real valuation $v$ is \emph{centered at} a scheme-theoretic
point $\xi=c_X(v)\in X$ if we have a local inclusion 
$\cO_{\xi,X}\hookrightarrow\cO_v$ of local rings.
Notice that the center of a valuation, if exists,
is unique since $X$ is separated. Denote by $\Val_X$ 
the set of real valuations of $K(X)$ that admits a center
on $X$. For a closed point $x\in X$, we denote by $\Val_{X,x}$ the set
of real valuations of $K(X)$ centered at $x\in X$. It's well known that $v\in \Val_{X}$ is centered at $x\in X$ if $v(f)$ for any $f\in \fm_x$.

For each valuation $v\in \Val_{X,x}$ and any integer $m$, we define
the valuation ideal $\fa_m(v):=\{f\in\cO_{x,X}\mid v(f)\geq m\}$. Then
it is clear that $\fa_m(v)$ is an $\fm_x$-primary ideal for each $m>0$.

Given a valuation $v\in \Val_X$ and a nonzero ideal $\fa\subset\cO_X$, we may evaluate $\fa$ along $v$ by setting $v(\fa) := \min\{v(f)\mid f \in \fa\cdot\cO_{c_X(v),X} \}$. 
It follows from the above definition that if $\fa\subset \fb \subset \cO_X$ are nonzero ideals, then $v(\fa) \geq v(\fb)$.
Additionally, $v(\fa)> 0$ if and only if $c_X(v) \in \Cosupp(\fa)$.
We endow $\Val_X$  with the weakest topology such that,
for every ideal $\fa$ on $X$, the map $\Val_X\to \bR\cup\{+\infty\}$ defined by $v\mapsto v(\fa)$ is continuous.
The subset $\Val_{X,x}\subset \Val_X$ is endowed with
the subspace topology. In some literatures, the space  $\Val_{X,x}$ is called the {\it non-archimedean link} of $x\in X$. When $X=\bC^2$, the geometry of $\Val_{X,x}$ is understood well (see \cite{FJ04}). For higher dimension, its structure is much more complicated but can be described as an inverse limit of dual complexes (see \cite{JM12, BdFFU15}).

\bigskip

Let $Y\xrightarrow[]{\mu} X$ be a proper birational morphism with $Y$  a normal variety. 
For a prime divisor $E$ on $Y$, we define a valuation $\ord_E\in \Val_X$ that sends each rational function in $K(X)^{\times}=K(Y)^{\times}$ to its order of vanishing along $E$. Note that 
the center $c_X(\ord_E)$ is the generic point of $\mu(E)$.
We say that $v\in \Val_X$ is a \emph{divisorial valuation} if there exists $E$ as above and $\lambda\in\bR_{>0}$ such that $v=\lambda\cdot \ord_E$.

Let $\mu : Y \to X$ be a proper 
birational morphism and $\eta\in Y$ a point 
such that $Y$ is regular at $\eta$. Given a system of
parameters $y_1,\cdots, y_r \in \cO_{Y,\eta}$ at $\eta$
and $\alpha = (\alpha_1,\cdots,\alpha_r) \in \bR_{\geq 0}^r\setminus\{0\}$,
we define a valuation $v_\alpha$ as follows. For $f\in \cO_{Y,\eta}$
we can write it as $f =\sum_{\beta\in\bZ_{\geq 0}^r}c_\beta y^\beta$, with $c_\beta\in\widehat{\cO_{Y,\eta}}$  either zero or unit. We set
\[
v_\alpha(f) = \min\{\langle\alpha,\beta\rangle\mid c_\beta\neq 0\}.
\]
A \emph{quasi-monomial valuation} is a valuation that can be written in the above form. 

Let $(Y, E =\sum_{k=1}^N E_k)$ be a log smooth model of $X$, i.e. $\mu : Y \to X$ is an isomorphism outside of the support of $E$. We denote by $\QM_{\eta}(Y,E)$ the set of all
quasi-monomial valuations $v$ that can be described at the point $\eta\in Y$ with respect to
coordinates $(y_1,\cdots, y_r)$ such that each $y_i$ defines at $\eta$ an irreducible component of $E$ 
(hence $\eta$ is the generic point of a connected component of the intersection of some of the divisors $E_i$).
We put $\QM(Y,E):=\bigcup_{\eta}\QM_\eta(Y,E)$ where $\eta$ runs over
generic points of all irreducible components of intersections of 
some of the divisors $E_i$.

Given a valuation $v\in \Val_{X,x}$, its \emph{rational rank} $\ratrank v$
is the rank of its value group. The \emph{transcendental degree}
$\transdeg v$ of $v$ is the transcendental degree of the field extension
$\bC\hookrightarrow \cO_v/\fm_v$. The Zariski-Abhyankar Inequality
says that 
\[
 \transdeg v+\ratrank v\leq \dim X.
\]
A valuation satisfying the equality is called an \emph{Abhyankar valuation}.
By \cite{ELS03}, we know that a valuation $v\in\Val_{X}$ is Abhyankar
if and only if it is quasi-monomial.

\bigskip

\begin{defn}\label{d-vol}
 Let $X$ be an $n$-dimensional normal variety. Let $x\in X$ be a closed point. We define
 the \emph{volume of a valuation} $v\in\Val_{X,x}$ following \cite{ELS03} as
 \[
  \vol_{X,x}(v)=\limsup_{m\to\infty}\frac{\ell(\cO_{x,X}/\fa_m(v))}{m^n/n!}.
 \]
 where $\ell$ denotes the length of the artinian module.
\end{defn}
Thanks to the works of \cite{ELS03, LM09, Cut13} the above limsup is actually a limit.

\begin{defn}\label{d-logd}
 Let $(X,D)$ be a klt log pair. We define the \emph{log discrepancy
 function of valuations} $A_{(X,D)}:\Val_X\to (0,+\infty]$
 in successive generality.
 \begin{enumerate}[label=(\alph*)]
  \item Let $\mu:Y\to X$ be a proper birational morphism from
  a normal variety $Y$. Let $E$ be a prime divisor
  on $Y$. Then we define $A_{(X,D)}(\ord_E)$ as 
  \[
   A_{(X,D)}(\ord_E):=1+\ord_E(K_Y-\mu^*(K_X+D)).
  \]
  \item Let $(Y,E=\sum_{k=1}^N E_k)$ be a log smooth model of $X$.
  Let $\eta$ be the generic point of a connected component of 
  $E_{i_1}\cap E_{i_2}\cap\cdots\cap E_{i_r}$ of codimension $r$. Let $(y_1,\cdots,y_r)$
  be a system of parameters of $\cO_{Y,\eta}$ at $\eta$ such that
  $E_{i_j}=(y_j=0)$. Then
  for any $\alpha=(\alpha_1,\cdots,\alpha_r)\in\bR_{\geq 0}^r\setminus\{0\}$, we define $A_{(X,D)}(v_{\alpha})$ as
  \[
   A_{(X,D)}(v_\alpha):=\sum_{j=1}^r \alpha_j A_{(X,D)}(\ord_{E_{i_j}}).
  \]
  \item In \cite{JM12}, it was showed that there exists a retraction map $r_{Y,E}: \Val_X \to \QM(Y, E)$ for any log smooth model $(Y, E)$ over $X$, such that it induces a
homeomorphism $\Val_X \to\varprojlim_{(Y,E)}\QM(Y, E)$. For any real valuation $v \in \Val_X$, we define
\[
A_{(X,D)}(v):=\sup_{(Y,E)} A_{(X,D)}(r_{(Y,E)}(v)).
\]
where $(Y, D)$ ranges over all log smooth models over $X$. For details, see \cite{JM12} and
\cite[Theorem 3.1]{BdFFU15}. It is possible that $A_{(X,D)}(v) = +\infty$ for some $v\in \Val_X$, see e.g. \cite[Remark 5.12]{JM12}.
 \end{enumerate}
\end{defn}

Then we can define the main invariant in this paper. As we mentioned in Section \ref{s-sasaki}, it is partially inspired the definition in \cite{MSY08} for a valuation coming from the Reeb vector field.

\begin{defn}[\cite{Li15}]\label{d-normvol}
 Let $(X,D)$ be an $n$-dimensional klt log pair. Let $x\in X$ be a closed point.
 Then the \emph{normalized volume function of valuations} $\hvol_{(X,D),x}:\Val_{X,x}\to(0,+\infty)$
 is defined as
 \[
  \hvol_{(X,D),x}(v)=\begin{cases}
            A_{(X,D)}(v)^n\cdot\vol_{X,x}(v), & \textrm{ if }A_{(X,D)}(v)<+\infty;\\
            +\infty, & \textrm{ if }A_{(X,D)}(v)=+\infty.
           \end{cases}
 \]
 The \emph{volume of the singularity} $(x\in (X,D))$ is defined as
 \[
  \hvol(x, X,D):=\inf_{v\in\Val_{X,x}}\hvol_{(X,D),x}(v).
 \]

\end{defn}
Since $\hvol(v)=\hvol(\lambda\cdot v)$ for any $\lambda \in \bR_{>0}$, for any valuation $v\in \Val_{X,D}$ with a finite log discrepancy, we can rescale such that $\lambda\cdot v\in \Val^{=1}_{X,D}$ where 
$\Val^{=1}_{X,D}$ consists of all valuations $v\in \Val_{X,x}$ with $ A_{(X,D)}(v)=1$. 

\begin{rem}A definition of volume of singularities is also given in \cite{BdFF12}. Their definition is the local analogue of the volume $K_X$ whereas our definition is the one of the volume of $-K_X$. In particular, a singularity has volume 0 in the definition of \cite{BdFF12} if it is log canonical. 
\end{rem}

\subsection{Properties} In this section, we discuss some properties of $\hvol$ on $\Val_{X,x}$. 
We start from the properness and Izumi estimates. As a corollary, we conclude that $\hvol(x, X,D)$ is always positive for any klt singularity $x\in (X,D)$. 

\begin{thm}[\cite{Li15}]\label{t-izumi}
 Let $(x\in(X,D))$ be a klt singularity. Then there exists positive constants $C_1$, $C_2$ which only depend on $x\in (X,D)$ (but not the valuation $v$) such that the following holds.
\begin{enumerate}
\item (Izumi-type inequality) For any valuation $v\in\Val_{X,x}$, we have
\[
v(\fm_x)\ord_x\leq v\leq C_2\cdot A_{(X,D)}(v)\ord_x.
\]
\item (Properness) For any valuation $v\in\Val_{X,x}$ with $A_{(X,D)}(v)<+\infty$, we have
\[
 C_1\frac{A_{(X,D)}(v)}{v(\fm_x)} \le \hvol(v).
\]
\end{enumerate}
\end{thm}
Note that since $x\in X$ is singular, $\ord_x$ in the above inequality might not be a valuation. In other words, for $f,g\in \cO_{X,x}$, $\ord_x(fg)\ge \ord_x(f)+\ord_x(g)$ may be a strict inequality.

The above Izumi type inequality is well known when $x\in X$ is a smooth point. In the case of a general  klt singularity,  it can be reduced to the smooth case after a log resolution and decreasing the constant. Then for the properness, it follows from a more subtle estimate that there exists a positive constant $c_2$ that
\[
\vol(v)\ge c_2\left(\sup_{\fm_x} \frac{v}{\ord_x}\right)^{1-n} \cdot \frac{1}{v(\fm)}.
\]

\bigskip
Let $\fa_\bullet=\{\fa_m\}_{m\in\bZ}$ be a graded sequence of $\fm_x$-primary ideals. By the works in \cite{LM09, Cut13}, the following identities hold true:
\begin{equation*}
\mult(\fa_\bullet):=\lim_{m\rightarrow+\infty}\frac{\ell(\cO_{X,x}/\fa_m)}{m^n/n!}=\lim_{m\rightarrow+\infty}\frac{\mult(\fa_m)}{m^n}.
\end{equation*}
In particular, the two limits exist. Note that, by definition, for any $v\in \Val_{X,x}$ and $\fa_\bullet(v)=\{\fa_m(v)\}$, we have $\vol(v)=\mult(\fa_\bullet(v))$ .

The following observation on characterizing the normalized volumes by normalized 
multiplicities provides lots of flexibility in the study as we will see. \begin{thm}[\cite{Liu16}]\label{thm:liueq}
 Let $(x\in (X,D))$ be an $n$-dimensional klt singularity. Then we have
 \[
  \hvol(x,X,D)=\inf_{\fa\colon \fm_x\textrm{-primary}}\lct(X,D;\fa)^n\mult(\fa)
  =\inf_{\fa_\bullet\colon \fm_x\textrm{-primary}}\lct(X,D;\fa_\bullet)^n\mult(\fa_\bullet).
 \]
We also set $\lct(X,D;\fa_\bullet)^n\mult(\fa_\bullet)=
+\infty$ if $\lct(X,D;\fa_\bullet)=+\infty$.
\end{thm}

\begin{proof}
 Firstly, for any $\fm_x$-primary ideal $\fa$, we can take
 a divisorial valuation $v\in\Val_{X,x}$ computing $\lct(\fa)$. In other
 words, $\lct(\fa)=A_X(v)/v(\fa)$. We may rescale $v$
 such that $v(\fa)=1$. Then clearly $\fa^m\subset\fa_m(v)$
 for any $m\in\bN$, hence $\mult(\fa)\geq \vol(v)$.
 Therefore, $\lct(\fa)^n\mult(\fa)\geq A_X(v)^n\vol(v)$ which implies
 \begin{equation}\label{eq:liueq1}
  \hvol(x, X,D)\leq \inf_{\fa\colon \fm_x\textrm{-primary}}\lct(\fa)^n\mult(\fa).
 \end{equation}
 
 Secondly, for any graded sequence of $\fm_x$-primary 
 ideals $\fa_\bullet$, we have
 \[
  \lct(\fa_\bullet)=\lim_{m\to\infty}m\cdot\lct(\fa_m)
 \]
 by \cite{JM12, BdFFU15}. Hence 
 \[
  \lct(\fa_\bullet)^n\mult(\fa_\bullet)=
  \lim_{m\to\infty} (m\cdot\lct(\fa_m))^n\frac{\mult(\fa_m)}{m^n}
  =\lim_{m\to\infty}\lct(\fa_m)^n\mult(\fa_m).
\]
As a result,
\begin{equation}\label{eq:liueq2}
 \inf_{\fa\colon \fm_x\textrm{-primary}}\lct(\fa)^n\mult(\fa)
  \leq \inf_{\fa_\bullet\colon \fm_x\textrm{-primary}}
  \lct(\fa_\bullet)^n\mult(\fa_\bullet).
\end{equation}

Lastly, for any valuation $v\in\Val_{X,x}$, we consider
the graded sequence of its valuation ideals $\fa_\bullet(v)$.
Since $v(\fa_\bullet(v))=1$, we have $\lct(\fa_\bullet)
\leq A_X(v)$. We also have $\mult(\fa_\bullet(v))=\vol(v)$.
Hence $\lct(\fa_{\bullet}(v))^n\mult(\fa_\bullet(v))\leq
A_X(v)^n\vol(v)$, which implies
\begin{equation}\label{eq:liueq3}
  \inf_{\fa_\bullet\colon \fm_x\textrm{-primary}}
  \lct(\fa_\bullet)^n\mult(\fa_\bullet)\le \hvol(x, X,D).
\end{equation}
The proof is finished by combining \eqref{eq:liueq1}, \eqref{eq:liueq2},
and \eqref{eq:liueq3}.
\end{proof}

In general we have the following relation between a sequence of graded ideals and the one from a valuation: Let $\Phi^g$ be an ordered subgroup of the real numbers $\bR$. Let $(R, \fm)$ be the local ring at a normal singularity $o\in X$. A $\Phi^g$-graded filtration of $R$, denoted by $\cF:=\{\fa^m\}_{m\in \Phi^g}$, is a decreasing family of $\fm$-primary ideals of $R$ satisfying the following conditions:

{\bf (i)} $\fa^m\neq 0$ for every $m\in \Phi^g$, $\fa^m=R$ for $m\le 0$ and $\cap_{m\ge 0}\fa^m=(0)$;

{\bf (ii)} $\fa^{m_1}\cdot \fa^{m_2}\subseteq \fa^{m_1+m_2}$ for every $m_1, m_2\in \Phi^g$.

Given such an $\cF$, we get an associated order function 
$$v=v_{\cF}: R\rightarrow \bR_{\ge 0} \qquad v(f)=\max\{m; f\in \fa^m\} \mbox{\ \ for any $f\in R$}.$$ Using the above {\bf (i)-(ii)}, it is easy to verify that $v$ satisfies $v(f+g)\ge \min\{v(f), v(g)\}$ and $v(fg)\ge v(f)+v(g)$. We also have the associated graded ring:
\[
\gr_{\cF}R=\sum_{m\in \Phi^g} \fa^m/\fa^{>m}, \text{ where } \fa^{>m}=\bigcup_{m'> m}\fa^{m'}.
\]
For any real valuation $v$ with valuative group $\Phi^g$, $\{\cF^m\}:=\{\fa_m(v)\}$ is a $\Phi^g$-graded filtration of $R$. 
 We will need the following facts.
\begin{lem}[see \cite{Tei03, Tei14}]\label{lem-quasi}
With the above notations, the following statements hold true:

{\rm (1) (\cite[Page 8]{Tei14})} If $\gr_{\cF}R$ is an integral domain, then $v=v_{\cF}$ is a valuation centered at $o\in X$. In particular, $v(fg)=v(f)+v(g)$ for any $f,g\in R$.

{\rm (2) (Piltant)} A valuation $v$ is quasi-monomial if and only if the Krull dimension of $\gr_v R$ is the same as the Krull dimension of $R$.
\end{lem}

The existence of a minimizer for $\hvol_{(X,D),x}$ was conjectured in the first version of \cite{Li15} and then proved in \cite{Blu18}. 

\begin{thm}[\cite{Blu18}]\label{t-existence}
 For any klt singularity $x\in(X,D)$,
 there exists a valuation $v_{\min}\in\Val_{X,x}$ that minimizes
 the function $\hvol_{(X,D),x}$.
\end{thm}

Let us sketch the idea of proving the existence of
$\hvol$-minimizer. We first take a sequence of valuations $(v_i)_{i\in\bN}$ such that
\[
 \lim_{i\to\infty}\hvol(v_i)=\hvol(x, X,D).
\]
Then we would like to find a valuation $v^*$ that is a limit point
of the sequence $(v_i)_{i\in\bN}$ and then show that $v^*$ is a minimizer of $\hvol$. 

Instead of seeking a limit point $v^*$ of $(v_i)_{i\in\bN}$ in the space of valuations, we consider graded sequences of ideals. More precisely, each valuation $v_i$ induces a graded sequence $\fa_\bullet(v_i)$ of $\fm_x$-primary ideals. By Theorem \ref{thm:liueq}, we have
\[
 \hvol(v_i)\geq \lct(\fa_\bullet(v_i))^n\mult(\fa_\bullet(v_i))\geq \hvol(x, X,D).
\]
Therefore, once we find a graded sequence of $\fm_x$-primary ideals $\tilde{\fa}_\bullet$ that is a `limit point' of the sequence $(\fa_{\bullet}(v_i))_{i\in\bN}$, a valuation $v^*$ computing $\lct(\tilde{\fa}_\bullet)$ will minimizes $\hvol$. The existence of such `limits' relies on two ingredients: the first is an asymptotic estimate to control the growth for $\fa_k(v_i)$ for a fixed $k$; once the growth is controlled, we can apply the generic limit construction. 

\begin{proof} For simplicity, we will assume $D=0$.
More details about log pairs can be found in \cite[Section 7]{Blu18}.

Let us choose a sequence of valuations
$v_i\in\Val_{X,x}$ such that $$\lim_{i\to\infty}\hvol(v_i)=\hvol(x, X).$$
Since the normalized volume function is invariant after rescaling, 
we may assume that $v_i(\fm)=1$ for all $i\in\bN$ where $\fm:=\fm_x$.
Our goal is to show that the family of graded sequences
of $\fm$-primary ideals $(\fa_\bullet(v_i))_{i\in\bN}$
satisfies the following conditions:
\begin{enumerate}[label=(\alph*)]
 \item For every $\epsilon>0$, there exists positive constants $M, N$ so that
\[
\lct(\fa_m(v_i))^n\mult(\fa_m(v_i))\leq \hvol(x, X)+\epsilon
\textrm{ for all }m\geq M\textrm{ and } i\geq N.
\]
\item For each $m,i\in\bN$, we have $\fm^m\subset\fa_m(v_i)$.
\item There exists $\delta>0$ such that $\fa_m(v_i)\subset\fm^{\lfloor m\delta\rfloor}$ for all $m,i\in\bN$.
\end{enumerate}

Part (b) follows easily from $v_i(\fm)=1$. Hence $\vol(v_i)\leq \mult(\fm)=:B$.
For part (c), we need to use Theorem \ref{t-izumi}. 
By Part (2),
there exists a positive constant $C_1$ such that 
\[
 A_X(v)\leq C_1^{-1}\cdot v(\fm)\hvol(v) \textrm{ for all }v\in\Val_{X,x}.
\]
Let $A:=C_1^{-1}\sup_{i\in\bN}\hvol(v_i)$, then $A_X(v_i)\leq A$
for any $i\in\bN$. By Theorem \ref{t-izumi}(1), then there exists a positive constant $C_2$
such that
\[
v(f)\leq C_2\cdot A_X(v)\ord_x(f)\textrm{ for all }v\in \Val_{X,x}\textrm{ and }f\in\cO_{X,x}.
\]
In particular, $v_i(f)\leq C_2 A\cdot\ord_x(f)$ for all $i\in\bN$ and $f\in\cO_{X,x}$.
Thus by letting $\delta:=(C_2 A)^{-1}$ we have 
$\fa_m(v_i)\subset\fm^{\lfloor m\delta\rfloor}$ which proves part
(c).

\bigskip

The proof of part (a) relies on the following result on uniform
convergence of multiplicities of valuation ideals.
\begin{prop}[\cite{Blu18}]\label{p-uniform}
 Let $(x\in X)$ be an $n$-dimensional klt singularity.
 Then for $\epsilon, A, B, r\in \bR_{>0}$, there exists
 $M=M(\epsilon, A, B, r)$ such that for every valuation
 $v\in\Val_{X,x}$ with $A_X(v)\leq A$, $\vol(v)\leq B$,
 and $v(\fm)\geq 1/r$, we have
\[
\vol(v)\leq \frac{\mult(\fa_m(v))}{m^n}<\vol(v)+\epsilon\textrm{ for all } m\geq M.
\]
\end{prop}
\begin{proof}The first inequality is straightforward. When the point is smooth, the second inequality uses the inequality that for the graded sequence of ideals $\{\fa_{\bullet}\}$, there exists a $k$ such that for any $m$ and $l$ 
$$\fa_{ml}\subseteq \fa^{l}_{m-k}.$$
The proof of such result uses the multiplier ideal, see \cite{ELS03}. For isolated klt singularity, then an estimate of a similar form in \cite{Tak06} says
\begin{eqnarray}\label{e-takagi}
\mathcal{J}_X^{l-1}\cdot \fa_{ml}\subset \fa^{l}_{m-k}
\end{eqnarray}
suffices, where $\mathcal{J}_X $ is the Jacobian ideal of $X$.  Finally, in the general case, an argument using \eqref{e-takagi} and interpolating $\mathcal{J}_X$ and a power of $\fm$ gives the proof. See \cite[Section 3]{Blu18} for more details.
\end{proof}
To continue the proof, let us fix an arbitrary $\epsilon\in\bR_{>0}$.
Since $A_X(v_i)\leq A$, $\vol(v_i)\leq B$, and 
$v_i(\fm)=1$ for all $i\in\bN$, Proposition \ref{p-uniform} implies
that there exists $M\in\bN$ such that 
\[
 \frac{\mult(\fa_m(v_i))}{m^n}\leq \vol(v_i)+\epsilon/(2A^n)\textrm{ for all }
 i\in\bN.
\]
We also have $\lct(\fa_m(v_i))\leq A_X(v_i)/v_i(\fa_m(v_i))\leq m\cdot A_X(v_i)$.
Let us take $N\in\bN$ such that $\hvol(v_i)\leq \hvol(x,X)+\epsilon/2$
for any $i\geq N$. Therefore,
\begin{align*}
\lct(\fa_m(v_i))^n\mult(\fa_m(v_i))&\leq A_X(v_i)^n(\vol(v_i)+\epsilon/(2A^n))\\
& =\hvol(v_i)+\epsilon\cdot A_X(v_i)^n/(2A^n)\\
& =\hvol(v_i)+\epsilon/2\\&\leq \hvol(x,X)+\epsilon.
\end{align*}
So part (a) is proved.

Finally, (b) and (c) guarantee that we can apply a generic limit type construction (cf. \cite[Section 5]{Blu18}). Then (a) implies that
 a `limit point' $\tilde{\fa}_\bullet$
of the sequence $(\fa_\bullet(v_i))_{i\in\bN}$
satisfies that $\lct(\tilde{\fa}_\bullet)^n\mult(\tilde{\fa}_\bullet)\leq\hvol(x,X)$.
Thus a valution $v^*$ computing the log canonical threshold
of $\tilde{\fa}_\bullet$, whose existence follows
from \cite{JM12}, necessarily minimizes the normalized volume.
\end{proof}

\begin{thm}[\cite{LiuX17}]\label{t-smooth}
 Let $x\in (X,D)$ be an $n$-dimensional klt singularity. Then
 $\hvol(x,X,D)\leq n^n$ and the equality holds if and only
 if $x\in X\setminus\Supp(D)$ is a smooth point.
\end{thm}

Using the fact that we can specialize a graded sequence of ideals preserving the colength, and the lower semi-continuous of the log canonical thresholds, we easily get the inequality part of Theorem \ref{t-smooth}. Then the equality part gives us a characterization of the smooth point using the normalized volume. The following Theorem \ref{t-lower} on the semicontinuity needs a more delicate analysis.  We conjecture that the normalized volume function is indeed constructible (see Conjecture \ref{c-constru}). 

\begin{thm}[\cite{BL18}]\label{t-lower}
 Let $\pi:(\cX,D)\to T$ together with a section $t\in T\mapsto x_t\in\cX_t$
 be a $\bQ$-Gorenstein flat family of klt singularities. Then
 the function $t\mapsto \hvol(x_t,\cX_t,D_t)$  is lower
 semicontinuous with respect to the Zariski topology.
\end{thm}

Now we introduce a key tool that the minimal model program provides to us to understand minimizing the normalized volume. For more discussions, see Section \ref{ss-general}.

\begin{defn}[Koll\'ar component, \cite{Xu14}]\label{d-kollar}
Let $x\in (X,D)$ be a klt singularity. We call a proper birational morphism $\mu:Y\to X$ provides a  Koll\'ar component $S$, if $\mu$ is isomorphic over $X\setminus \{x\}$, and $\mu^{-1}(x)$ is an irreducible divisor $S$, such that $(Y,S+\mu^{-1}_*D)$ is purely log terminal (plt) and $-S$ is $\mathbb{Q}$-Cartier and ample over $X$.
\end{defn}

\begin{thm}[{\cite{LX16}}]
We have the identity:
\begin{eqnarray}
\hvol(x, X, D)=\inf_S \{ \hvol(\ord_S)\ | \ \mbox{for all Koll\'ar components $S$ over }x\}. 
\end{eqnarray}
\end{thm}
For the explanation of proof, see the discussions for \eqref{e-kol} in Section \ref{ss-general}. 

\section{Stability in Sasaki-Einstein geometry}\label{s-sasaki}
To proceed the study of normalized volumes, we will introduce the concept of K-stability. This is now a central notion in complex geometry, which serves as an algebraic characterization of the existence of some `canonical metrics'. 

In the local setting, such problem on an affine $T$-variety $X$ with a unique fixed point $x$ was first considered in \cite{MSY08}. We can then varies the Reeb vector field $\xi\in \ft^+_{\bR}$, and call such a structure $(X,\xi)$ is a Fano cone if $X$ only has klt log terminal singularities. The name is justified since if $\xi\in \ft^+_{\bQ}$, let $\la\xi \ra$ be the $\bC^*$ generated by $\xi$, then $X\setminus \{x\}/\la\xi \ra$ is a log Fano variety.  

In \cite{MSY08}, the relation between the existence of Sasaki-Einstein metric along $(X,\xi)$ and the K-stability of $(X,\xi)$, a mimic of the absolute case, was explored. A key observation in \cite{MSY08} is that we can define a normalized volume function $\hvol_{X}(\xi)$ for $\xi \in \ft^+_{\mathbb{R}}$, and among all choices of $\xi$ the one minimizing $\hvol_{X}(\cdot)$ gives `the most stable' direction.  

Then an important step to advance such a picture is made in \cite{CS18,CS15} by extending the definition of K-stability notions on $(X,\xi)$ allowing degenerations, and showing that there is a Sasaki-Einstein metric along an isolated Fano cone singularity $(X,\xi)$ if and only of $(X,\xi)$ is K-polystable, extending the solution of the Yau-Tian-Donaldson's conjecture in the Fano manifold case (see \cite{CDS,Tia15}) to the cone case. 

In this section, we will briefly introduce these settings.

\subsection{T-varieties}\label{s-tvariety}
We first introduce the basic setting using $T$-varieties. 
For general results of $T$-varieties,  see  \cite{AIPSV12}. 

Assume $X={\rm Spec}_{\bC}(R)$ is an affine variety with $\bQ$-Gorenstein klt singularities. Denote by $T$ the complex torus $(\bC^*)^r$. Assume $X$ admits a good $T$-action in the following sense. 
\begin{defn}[{see \cite[Section 4]{LS13}}]\label{d-good}
Let $X$ be a normal affine variety. We say that a $T$-action on $X$ is {\it good} if 
it is effective and there is a unique closed point $x\in X$ that is in the orbit closure of any $T$-orbit. We shall call $x$ (sometimes also denoted by $o_X$) the vertex point of the $T$-variety $X$.
\end{defn}

Let $N={\rm Hom}(\mathbb{C}^*, T)$ be the co-weight lattice and $M=N^*$ the weight lattice. We have a weight space decomposition of the coordinate ring of $X$:
\[
R=\bigoplus_{\alpha\in \Gamma} R_\alpha \mbox{  where \ } \Gamma =\{ \alpha\in  M  |\ R_{\alpha}\neq 0\}.
\]
The action being good implies $R_0=\mathbb{C}$, which will always be assumed in the below. An ideal $\ka$ is called homogeneous if $\fa=\bigoplus_{\alpha\in \Gamma}\ka\cap R_\alpha$. Denote by $\sigma^{\vee}\subset M_{\mathbb{Q}}$ the cone generated by $\Gamma$ over $\mathbb{Q}$, which will be called the {\it weight cone} or the {\it moment cone}.  
The cone $\sigma\subset N_{\bR}$, dual to $\sigma^\vee$, is the same as the following conical set 
$$\mathfrak{t}^+_{\bR}:=\{\ \xi \in N_{\mathbb{R}}\ \ | \  \langle \alpha, \xi \rangle>0 \mbox{ for any }\alpha\in \Gamma\setminus\{0\} \}.$$

Motivated by notations from Sasaki geometry, we will introduce:
\begin{defn}\label{defn-Reeb}
With the above notations, $\frak{t}^+_\bR$ will be called the Reeb cone of the $T$-action of $X$. A vector $\xi\in \ft^+_\bR$ will be called a Reeb vector on the $T$-variety $X$. 
\end{defn}

To adapt this definition into our setting in Section \ref{ss-firstdefinition}, for any $\xi \in \mathfrak{t}^+_{\bR}$, we can define a valuation 
$$\wt_{\xi}(f) = \min_{\alpha\in \Gamma}\{\langle \alpha,\xi \rangle \ | \ f_{\alpha}\neq  0\}.$$
It is easy to verify that $\wt_\xi\in \Val_{X,o_X}$. The rank of $\xi$, denoted by ${\rm rk}(\xi)$, is the dimension of the subtorus $T_\xi$ (as a subgroup of $T$) generated by $\xi\in \ft$.
The following lemma can be easily seen.
\begin{lem}\label{lem-Tqmv}
For any $\xi\in \mathfrak{t}^+_{\bR}$, $\wt_\xi$ is a quasi-monomial valuation of rational rank equal to the rank of $\xi$. Moreover, the center of $\wt_\xi$ is $o_X$.
\end{lem}


We recall the following structure results for any $T$-varieties.
\begin{thm}[{see \cite[Theorem 4]{AIPSV12}}]
Let $X={\rm Spec}(R)$ be a normal affine variety and suppose $T={\rm Spec}\left(\bC[M]\right)$ has a good action on $X$ with the weight cone $\sigma^{\vee} \subset M_{\bQ}$. Then there exist a normal projective variety $Y$ and a polyhedral divisor $\fD$ such that there is an isomorphism of graded algebras:
\[
R\cong H^0(X, \mathcal{O}_X)\cong \bigoplus_{u\in \sigma^{\vee} \cap M} H^0 \big(Y, \mathcal{O}(\fD(u))\big)=: R(Y, \fD).
\] 
In other words, $X$ is equal to ${\rm Spec}_\bC\big( \bigoplus_{u\in \sigma^{\vee} \cap M} H^0(Y, \mathcal{O}(\fD(u)) ) \big)$.
\end{thm}
In the above definition, a polyhedral divisor $\fD: u\rightarrow \fD(u)$ is a map from $\sigma^\vee$ to the set of $\bQ$-Cartier divisors that satisfies:
\begin{enumerate}
\item $\fD(u)+\fD(u')\le \fD(u+u')$ for any $u, u'\in \sigma^\vee$;
\item $u\mapsto\fD(u)$ is piecewisely linear;
\item $\fD(u)$ is semiample for any $u\in \sigma^\vee$, and $\fD(u)$ is big if $u$ is in the relative interior of $\sigma^\vee$.
\end{enumerate}

Here $Y $ is projective since from our assumption $$H^0(Y,\mathcal{O}_Y)=R^T=R_0=\mathbb{C}$$
(see \cite{LS13}).
We collect some basic results about valuations on $T$-varieties.

\begin{thm}[see \cite{AIPSV12}]\label{t-Tcano}
Assume a $T$-variety $X$ is determined by the data $(Y, \sigma, \fD)$ such that $Y$ is a projective variety, where $\sigma=\ft^+_{\bR} \subset N_{\bR}$ and $\fD$ is a polyhedral divisor.\begin{enumerate}
\item
For any $T$-invariant quasi-monomial valuation $v$, there exist a quasi-monomial valuation $v^{(0)}$ over $Y$ and $\xi\in M_{\bR}$ such that for any $f\cdot \chi^u\in R_u$, we have:
\[
v(f\cdot \chi^u)=v^{(0)}(f)+\langle u, \xi\rangle.
\]
We will use $(\xi, v^{(0)})$ to denote such a valuation. 

\item
$T$-invariant prime divisors on $X$ are either vertical or horizontal. Any vertical divisor is determined by a divisor $Z$ on $Y$ and a vertex $v$ of $\fD_Z$, and will be denoted by $D_{(Z,v)}$. Any horizontal divisor is determined by a ray $\rho$ of $\sigma$ and will be denoted by $E_\rho$.

\item Let $D$ be a $T$-invariant vertical effective $\mathbb{Q}$-divisor.  If $K_X+D$ is $\bQ$-Cartier, then the log canonical divisor has a representation $K_X+D=\pi^*H+{\rm div}(\chi^{-u_0})$ where $H=\sum_Z a_Z\cdot Z$ is a principal $\bQ$-divisor on $Y$ and $u_0\in M_{\bQ}$. Moreover, the log discrepancy of the horizontal divisor $E_\rho$ is given by:
\begin{equation}
A_{(X,D)}(E_\rho)=\langle u_0, n_\rho\rangle,
\end{equation}
where $n_\rho$ is the primitive vector along the ray $\rho$.

\end{enumerate}

\end{thm}
\begin{proof}[Sketch of the proof]
For the first statement, the case of divisorial valuations follows from \cite[Section 11]{AIPSV12}. It can be extended to the case of quasi-monomial valuations by the same proof. Note also that any $T$-invariant quasimonomial valuation can be approximated by a sequence of $T$-invariant divisorial valuations. The second statement is in \cite[Proposition 3.13]{PS08}.  The absolute case (e.g. without boundary divisor $D$) for the third statement is from \cite[Section 4]{LS13} whose proof also works for the case of log pairs..
 \end{proof}

We will specialize the study of general affine $T$-varieties to case that the log pair is klt.
Assume $X$ is a normal affine variety with $\bQ$-Gorenstein klt singularities and a good $T$-action. Let $D$ be a $T$-invariant vertical divisor. Then there is a nowhere-vanishing $T$-equivariant section $s$ of $m(K_X+D)$ where $m$ is sufficiently divisible. The following lemma says that the log discrepancy of $\wt_{\xi}$ can indeed be calculated in a similar way as in the toric case (the toric case is well-known). Moreover, it can be calculated by using the weight of $T$-equivariant pluri-log-canonical sections. The latter observation was first made in \cite{Li15}.
\begin{lem}\label{lem-ldwt}
Using the same notion as in the Theorem \ref{t-Tcano}, the log discrepancy of $\wt_{\xi}$ is given by:
$A_{(X,D)}(\wt_\xi)=\langle u_0, \xi \rangle$. Moreover, let $s$ be a $T$-equivariant nowhere-vanishing holomorphic section of $|-m(K_X+D)|$, and denote $\cL_\xi$  the Lie derivative with respect to the holomorphic vector field associated to $\xi$. Then $A_{(X,D)}(\xi)=\lambda$ if and only if
\[
\mathcal{L}_{\xi}(s)=m \lambda s \quad \text{ for } \quad \lambda>0.
\]
\end{lem}
As a consequence of the above lemma, we can formally extend $A_{(X,D)}(\xi)$ to a linear function on $\ft_{\bR}$:
\begin{equation}\label{eq-linearA}
A_{(X,D)}(\eta)=\langle u_0, \eta\rangle.
\end{equation}
for any $\eta\in \ft_{\bR}$. By Lemma \ref{lem-ldwt}, $A_{(X,D)}(\eta)=\frac{1}{m} \cL_\eta s/s$ where $s$ is a $T$-equivariant nowhere-vanishing holomorphic section of $|-m(K_X+D)|$.

\begin{defn}[Log Fano cone singularity]\label{d-logfanocone}
Let $(X, D)$ be an affine pair with a good $T$ action. Assume $(X,D)$ is a normal pair with  klt singularities. Then for any $\xi\in \ft^+_\bR$,  we call the triple $(X,D,\xi)$ a {\it log Fano cone} structure that is polarized by $\xi$. We will denote by $\la \xi\ra$ the sub-torus of $T$ generated by $\xi$. If $\la \xi\ra\cong \bC^*$, then we call $(X, D, \xi)$ to be quasi-regular. Otherwise, we call it irregular.
\end{defn}

In the quasi-regular case, we can take the quotient $(X\setminus\{x\},D\setminus\{x\})$ by the $\mathbb{C^*}$-group generated by $\xi$ in the sense of Seifert $\bC^*$-bundles (see \cite{Kol04}), and we will denote by $(X,D)/\langle \xi \rangle$, which is a log Fano variety, because of the assumption that $(X,D)$ is klt at $x$ (see \cite[Lemma 3.1]{Kol13}).

\subsection{K-stability} In this section, we will discuss the K-stability notion of log Fano cones, which generalizes the K-stability of log Fano varieties originally defined by Tian and Donaldson. For irregular Fano cones, such a notion was first defined in \cite{CS18}. 


\begin{defn}[Test configurations]\label{defn-QGTC}
Let $(X, D, \xi_0)$ be a log Fano cone singularity and $T$ a torus containing $\la\xi_0\ra$. 

A $T$-equivariant test configuration (or simply called a test configuration) of $(X, D, \xi_0)$ is a quadruple $(\cX, \cD, \xi_0; \eta)$ with a map $\pi: (\cX, \cD)\rightarrow\bC$ satisfying the following conditions:
\begin{enumerate}
\item[(1)]
$\pi:\cX\rightarrow \bC$ is a flat family and $\cD$ is an effective $\mathbb{Q}$-divisor such that $\cD$ does not contain any component $X_0$, the fibres away from $0$ are isomorphic to $(X, D)$ and $\cX={\rm Spec}(\cR)$ is affine, where $\cR$ is a finitely generate flat $\bC[t]$ algebra. The torus $T$ acts on $\cX$, and we write $\mathcal{R}=\bigoplus_{\alpha}\mathcal{R}_{\alpha}$ as decomposition into weight spaces.

\item[(2)]
$\eta$ is a holomorphic vector field on $\cX$ generating a $\bC^*(=\langle \eta\rangle)$-action on $(\cX, \cD)$ such that $\pi$ is $\bC^*$-equivariant where $\bC^*$ acts on the base $\bC$ by the
multiplication (so that $\pi_*\eta=t\partial_t$ if $t$ is the affine coordinate on $\bC$) and there is a $\bC^*$-equivariant isomorphism $\phi: (\cX, \cD)\times_{\bC}\bC^*\cong
(X, D)\times \bC^*$.  

\item[(3)]
The torus $T$-action commutes with $\eta$. 
The holomorphic vector field $\xi_0$ on $\cX\times_{\bC}\bC^*$ (via the isomorphism $\phi$) extends to a holomorphic vector field on $\cX$ which we still denote to be $\xi_0$. 
\end{enumerate}

\noindent In most our study, we only need to treat the case that test configuration $(\cX, \cD, \xi_0; \eta)$ of $(X, D, \xi_0)$ satisfies  that 
\begin{enumerate}
\item[(4)] $K_{\cX}+\cD$ is $\bQ$-Cartier and the central fibre $(X_0, D_0)$ is klt
\end{enumerate}
In other words, we will mostly consider special test configurations (see \cite{LX14, CS15}). 
\end{defn}
\bigskip

Condition (1) implies that each weight piece $\cR_{\alpha}$ is a flat $\bC[t]$-module. So $X$ and $X_0$ have the same weight cone and Reeb cone with respect to the fiberwise $T$-action.

A test configuration $(\cX, \cD, \xi_0; \eta)$ is called a product one if there is a $T$-equivariant isomorphism $(\cX, \cD)\cong (X, D)\times \bC$ and $\eta=\eta_0+t\partial_t$ where $\eta_0$ is a holomorphic vector field on $X$ that preserves $D$ and commutes with $\xi_0$. In this case, we will denote $(\cX, \cD, \xi_0; \eta)$ by $$(X\times\bC, D\times \bC, \xi_0; \eta)=:(X_\bC, D_\bC, \xi_0; \eta).$$
 In \cite{MSY08}, only such test configurations are considered. 
 
 \begin{defn}[K-stability]\label{d-ksemiSE}For any special test configuration $(\cX, \cD, \xi_0; \eta)$ of $(X,D,\xi_0)$ with central fibre $(X_0, D_0, \xi_0)$, its generalized Futaki invariant is defined as
\begin{eqnarray*}
\Fut(\cX, \cD, \xi_0; \eta):=\frac{D_{-T_{\xi_0}(\eta)}\vol_{X_0}(\xi_0)}{\vol_{X_0}(\xi_0)} 
\end{eqnarray*}
where we denote
\begin{equation}\label{eq-normeta}
T_{\xi_0}(\eta)=\frac{A(\xi_0)\eta-A(\eta)\xi_0}{n}.
\end{equation}
Since the generalized Futaki invariant defined above only depends on the data on the central fibre, we will also denote it by $\Fut(X_0, D_0, \xi_0; \eta)$.

We say that $(X, D, \xi_0)$ is K-semistable, if for any special test configuration, $\Fut(\cX, \cD, \xi_0; \eta)$ is nonnegative.

We say that $(X, D, \xi_0)$ is K-polystable, if it is K-semistable, and any special test configuration $(\cX, \cD, \xi_0; \eta)$ with $\Fut(\cX, \cD, \xi_0; \eta)=0$ is a product test configuration. 
\end{defn}

In the above definition, we used the notation \eqref{eq-normeta} and the directional derivative:
$$
D_{-T_{\xi_0(\eta)}}\vol_{X_0}(\xi_0):=\left.\frac{d}{d\epsilon}\right|_{\epsilon=0}\vol_{X_0}(\xi_0-\epsilon T_{\xi_0(\eta)}).
$$
Recall that the $\pi_*\eta=t\partial_t$. Then the negative sign in front of $T_{\xi_0}(\eta)$ in the above formula is to be compatible with our later computation. 
Using the rescaling invariance of the normalized volume, it is easy to verify that the following identity holds:
\begin{equation}\label{eq-dirhvol}
D_{-T_{\xi_0}\eta}\vol_{X_0}(\xi_0)=\left.\frac{d}{d\epsilon}\right|_{\epsilon=0}\hvol_{X_0}(\wt_{\xi_0-\epsilon\eta})\cdot \frac{1}{n A(\xi_0)^{n-1}},
\end{equation}
where $A(\xi_0)=A_{(X_0,D_0)}(\wt_{\xi_0})$. As a consequence, we can rewrite the Futaki invariant of a special test configuration as:
\begin{equation}\label{eq-Futhvol}
\Fut(\cX, \cD, \xi_0; \eta):=D_{-\eta}\hvol_{X_0}(\wt_{\xi_0})\cdot \frac{1}{n A(\xi_0)^{n-1} \cdot \vol_{X_0}(\xi_0)}.
\end{equation}

One can show that, up to a constant, the above definition of $\Fut(\cX, \cD, \xi_0; \eta)$ coincides  with the one in \cite{CS18, CS15} defined using index characters.
For convenience of the reader, we recall their definition. It is enough to define the Futaki invariant for the central fibre which we just denote by $X$. For any $\xi\in \frak{t}_\bR^+$, the index character $F(\xi, t)$ is defined by:
\begin{equation}
F(\xi, t):=\sum_{\alpha\in \Gamma}e^{-t\langle \alpha, \xi\rangle}\dim_{\bC}R_\alpha.
\end{equation}
Then there is a meromorphic expansion for $F(\xi, t)$ as follows:
\begin{equation}
F(\xi, t)=\frac{a_0(\xi)(n-1)!}{t^n}+\frac{a_1(\xi)(n-2)!}{t^{n-1}}+O(t^{2-n}).
\end{equation}
One always has the identity $a_0(\xi)=\vol(\xi)/(n-1)!$.
\begin{defn}[see \cite{CS18}]
For any $\eta\in \frak{t}_{\bR}$, define:
\begin{eqnarray*}
\Fut_{\xi_0}(X, \eta)&=&\frac{1}{n-1}D_{-\eta}(a_1(\xi_0))-\frac{1}{n}\frac{a_1(\xi_0)}{a_0(\xi_0)}D_{-\eta} a_0(\xi_0)\\
&=&\frac{a_0(\xi_0)}{n-1}D_{-\eta} \left(\frac{a_1}{a_0}\right)(\xi_0)+\frac{a_1(\xi_0)D_{-\eta} a_0(\xi_0)}{n(n-1)a_0(\xi_0)}.
\end{eqnarray*}
\end{defn}
This is a complicated expression. But in \cite[Proposition 6.4]{CS15}, it was showed that, when $X$ is $\bQ$-Gorenstein log terminal, there is an identity $a_1(\xi)/a_0(\xi)=A(\xi)(n-1)/2$ for any $\xi\in \frak{t}_{\bR}^+$ (by using our notation involving log discrepancies).  Note that the rescaling properties $a_0(\lambda \xi)=\lambda^{-n}a_0(\xi)$ and $a_1(\lambda \xi)=\lambda^{-(n-1)}a_1(\xi)$ which imply $\Fut_{\xi_0}(X, \xi_0)=0$. If we denote $\eta'=\eta-\frac{A(\eta)}{A(\xi_0)}\xi_0$, then we get:
\begin{equation}
\Fut_{\xi_0}(X, \eta)=\Fut_{\xi_0}\left(X, \eta'\right)=\frac{A(\xi_0)}{2n}D_{-\eta'}a_0(\xi_0)=\frac{1}{2(n-1)!} D_{-T_{\xi_0}(\eta)}\vol(\xi_0).
\end{equation}  
So the definition in \cite{CS18, CS15} differs from our notation by a constant $2(n-1)!/\vol_{X}(\xi_0)$. 
\begin{rem}
More precisely, our notation differs from that in \cite{CS18} by a sign. Our choice of minus sign for $-\eta$, besides being compatible with the sign choice in Tian's original definition of K-stability in \cite{Tia97}, is made for least two reasons. The first is that the careful calculation in \cite[Section 5.2]{LX17} shows that the limiting slope of the
Ding energy along the geodesic ray associated to any special test configuration is indeed the directional derivative of $\vol(\xi)$ along $-\eta$ instead of $\eta$. For the second reason, as we stressed in \cite[Remark 3.4]{LX17}, for the special test configuration coming from a Koll\'{a}r component $S$, the $-\eta$ vector corresponds to $\ord_S$. Since our goal is to compare $\hvol(\wt_{\xi_0})$ and $\hvol(\ord_S)$, $-\eta$ is the correct choice of sign (see \cite[Proof of Theorem 3.5]{LX17}). 
\end{rem}

\begin{rem}\label{r-Ding} 
In fact, in a calculation, instead of the generalized Futaki invariant, it is the Berman-Ding invariant,  denoted by $D^\NA(\cX, \cD, \xi_0; \eta)$, where
\begin{eqnarray*}
D^\NA(\cX, \cD, \xi_0; \eta):=\frac{D_{-T_{\xi_0}(\eta)}\vol_{X_0}(\xi_0)}{\vol(\xi_0)}-(1-\lct(\cX, D; \cX_0)).
\end{eqnarray*}
appears more naturally, whenever we know 
\begin{enumerate}
\item[{\bf (D)}] there exists a nowhere vanishing section $s\in |m (K_{\cX}+\cD)|$ such that we can use it to define $A(\cdot)$ as in the formula in Lemma \ref{lem-ldwt}.
\end{enumerate}
Then we can similarly define Ding semi(poly)-stable, replacing $\Fut(\cX, \cD, \xi_0; \eta)$ by $D^\NA(\cX, \cD, \xi_0; \eta)$. For a special test configuration, since
$$\Fut(\cX, \cD, \xi_0; \eta)=D^\NA(\cX, \cD, \xi_0; \eta)$$
the two notions coincide. 
\end{rem}

If we specialize the above definitions to the case of quasi-regular log Fano cone $(X,D,\xi_0)$, then we get the corresponding more familiar notions for the log Fano projective pair $(S,B)=(X,D)/\la \xi_0\ra$.

\subsection{Sasaki-Einstein geometry}

The introduction of normalized volumes in \cite{Li15} was motivated by the minimization phenomenon in the study of Sasaki-Einstein metrics. The latter was discovered in \cite{MSY06, MSY08} and  was motivated by the so called AdS/CFT correspondence from mathematical physics. Here we give a short account on this. For the reader who are mostly interested in the algebraic part of the theory, one can skip this section. The results will only be used in Section \ref{ss-dsc}.

Classically, a Sasaki manifold is defined as an odd dimensional Riemannian manifold $(M^{2n-1}, g_M)$ such that metric cone over it, defined as: 
$$(X, g_X):=((M\times \bR_{>0})\cup \{o_X\}, dr^2+r^2 g_M)$$ is K\"{a}hler. It's convenient to work directly on $X=X^\circ\cup o_X$ which is an affine variety with the K\"{a}hler metric $\sddb r^2$.  The Reeb vector field of $(X^\circ, g_X)$ is usually defined as $J(r\partial_r)$ where $J$ is the complex structure on $X^\circ$. The corresponding holomorphic vector field $\xi=r\partial_r-iJ(r\partial_r)$, which we also call Reeb vector field, generates a $T_\xi\cong (\bC^*)^{{\rm rk}(\xi)}$-action on $X$ where $r(\xi)\ge 1$. For simplicity, we will denote such a torus by $\langle \xi\rangle$. Moreover the corresponding element in $(\frak{t}_\xi)_\bR$, also denoted by $\xi$ is in the Reeb cone: $\xi\in (\frak{t}_\xi)^+_\bR$. The volume of $\xi$ is defined to be the volume density of $g_X$:
\begin{eqnarray}\label{eq-volxi}
\vol(\xi)&:=&\vol(r^2)=\frac{1}{(2\pi)^n n!}\int_X e^{-r^2}(\sddb r^2)^n\nonumber \\
&=&\frac{1}{(2\pi)^n}\int_M (-Jdr)\wedge (-dJd r)^{n-1}\nonumber \\
&=&\frac{(n-1)!}{2\pi^n}\vol(M, g_M)=\frac{\vol(M, g_M)}{\vol(\bS^{2n-1})}\nonumber \\
&=&\frac{\vol(B_1(X), g_X)}{\vol(B_1(\underline{0}), g_{\bC^n})}.
\end{eqnarray}
Here $g_X=\frac{1}{2}\sddb r^2(\cdot, J\cdot)$ and $g_M=\left.g_X\right|_M$ are the Riemannian metric on $X$ and $M$ respectively, 
$\bS^{2n-1}$ is the standard unit sphere in $\bC^n$ with volume $\vol(\bS^{2n-1})=2\pi^n/(n-1)!$.

This is well-defined because if two Sasaki metrics have the same Reeb vector field, then their volumes are the same. Indeed, $\omega_1=\sddb r_1$ and $\omega_2=\sddb r_2$ have the same Reeb vector field if $r_2=r_1 e^{\vphi}$ for a function $\vphi$ satisfying $\cL_{r\partial_r}\vphi=\cL_{\xi}\vphi=0$ (i.e. $\vphi$ is a horizontal function on $M$ with respect to the foliation defined by ${\rm Im}(\xi_0)$). Letting $r_t^2=r^2 e^{t\vphi}$ and differentiating the volume we get:
\begin{eqnarray*}
C\cdot \frac{d}{dt}\vol(r_t^2)&=&\int_X e^{-r_t^2}(-r_t^2\vphi)\sddb r_t^2)^n+e^{-r_t^2}n \sddb (r_t^2\vphi)\wedge  (\sddb r_t^2)^{n-1}\\
&=&\int_X -e^{-r_t^2} r_t^2\vphi (\sddb r_t^2)^n+e^{-r_t^2} n \sqrt{-1}\partial r_t^2\wedge (\vphi\bar{\partial}  r_t^2) \wedge (\sddb r_t^2)^{n-1}\\
&&\quad +\int_{X}e^{-r_t^2} n \sqrt{-1}\partial r_t^2\wedge (r_t^2 \bar{\partial}\vphi) \wedge (\sddb r_t^2)^{n-1}\\
&=&0.
\end{eqnarray*}
The second equality follows from integration by parts. The last equality follows by substituting $f=r_t^2$ and $f=\vphi$ in to the following identities and using the fact that $\vphi$ is horizontal (so that $\xi_t(\vphi)=0$):
\begin{eqnarray*}
n \sqrt{-1} \partial r_t^2 \wedge \bar{\partial} f\wedge (\sddb r_t^2)^{n-1}=\xi_t(f)(\sddb r_t^2)^n.
\end{eqnarray*}

One should compare this to the fact that two K\"{a}hler metrics in the same K\"{a}hler class have the same volume.

The Reeb vector field associated to a Ricci-flat K\"{a}hler cone metric satisfies the minimization principle in \cite{MSY08}. To state it in general, we assume $X$ is a $T$-variety with the Reeb cone $\frak{t}^+_\bR$ with respect to $T$ and recall the variation formulas of volumes of Reeb vector fields from \cite{MSY08}. For any $\xi\in \frak{t}^+_\bR$, we can find a radius function $r: X\rightarrow \bR_+$ such that $\vol(\xi)$ is given by the formula \eqref{eq-volxi}. 

\begin{lem}
The first order derivative of $\vol_{X}(\xi)$ is given by:
\begin{equation}\label{eq-Dvol}
D\vol(\xi)\cdot \eta_1=\frac{1}{(2\pi)^n (n-1)!}\int_X \theta_1 e^{-r^2}(\sddb r^2)^n,
\end{equation}
where $\theta_i=\eta_i(\log r^2)$.
The second order variation of $\vol_{X}(\xi)$ is given by:
\begin{eqnarray*}
D^2\vol(\xi)(\eta_1, \eta_2)&=&\frac{n+1}{(2\pi)^n(n-1)!}\int_X \theta_1\theta_2 e^{-r^2} (\sddb r^2)^n.
\end{eqnarray*}
\end{lem}
Now we fix a $\xi_0\in \frak{t}^+_\bR$ and a radius function $r: X\rightarrow \bR_+$ (by using equivariant embedding of $X$ into $\bC^N$ for example), we define:
\begin{defn}
$PSH(X, \xi_0)$ is the set of bounded real functions $\vphi$ on $X^\circ$ that satisfies:
\begin{enumerate}
\item[(1)] $\vphi\circ \tau=\vphi$ for any $\tau\in \langle \xi_0 \rangle$, the torus generated by $\xi_0$;
\item[(2)] $r^2_\vphi:=r^2e^\vphi$ is a proper plurisubharmonic function on $X$. 
\end{enumerate}
\end{defn}
To write down the equation of Ricci-flat K\"{a}hler-cone equation, we fix a $T$-equivariant no-where vanishing section $s\in H^0(X, mK_X)$ as in the last section and define an associated volume
form on $X$:
\begin{equation}\label{eq-MARF}
dV_X:=\left((\sqrt{-1})^{mn^2} s\wedge \bar{s}\right)^{1/m}.
\end{equation}
\begin{defn}\label{defn-RFKC}
We say that $r^2_\vphi:=r^2 e^{\vphi}$ where $\vphi\in PSH(X, \xi_0)$ is the radius function of a Ricci-flat K\"{a}hler cone metric on $(X, \xi_0)$ if $\vphi$ is smooth on $X^{\rm reg}$ and there exists a positive constant $C>0$ such that
\begin{equation}\label{eq-RFKC}
(\sddb r^2_\vphi)^n=C\cdot dV,
\end{equation}
where the constant $C$ is equal to:
\begin{equation*}
C=\frac{\int_X e^{-r^2_\vphi}(\sddb r^2_\vphi)^n}{\int_X e^{-r^2_\vphi}dV_X}=\frac{(2\pi)^n n! \vol(\xi_0)}{\int_X e^{-r^2_\vphi}dV_X}.
\end{equation*}
\end{defn}
Motivated by standard K\"{a}hler geometry, one defines the Monge-Amp\`{e}re energy $E(\vphi)$ using either its variations or the explicit expression on the link $M:=X\cap \{r=1\}$:
\begin{eqnarray*}
\delta E(\vphi)\cdot \delta\vphi&=&-\frac{1}{(n-1)!(2\pi)^n \vol(\xi_0)}\int_X \delta\vphi e^{-r^2_\vphi}(\sddb r^2_\vphi)^n.
\end{eqnarray*}
Then the equation \eqref{eq-MARF} is the Euler-Lagrange equation of the following Ding-Tian-typed functional:
\begin{eqnarray*}
D(\vphi)&=&E(\vphi)-\log\left(\int_X e^{-r^2_\vphi}dV_X\right).
\end{eqnarray*}
This follows from the identity:
\begin{eqnarray*}
\delta D(\vphi)\cdot \delta\vphi&=&\frac{1}{(2\pi)^n(n-1)!\vol(\xi_0)}\int_X \delta\vphi e^{-r^2_\vphi}(\sddb r^2_\vphi)^n-\frac{\int_X r^2_\vphi \delta\vphi e^{-r^2_\vphi}dV_X}{\int_X e^{-r^2_\vphi}dV_X}\\
&=&n \int_X e^{-r^2_\vphi} \delta\vphi \left(\frac{(\sddb r^2_\vphi)^n}{(2\pi)^n n!\vol(\xi_0)}-\frac{dV_X}{\int_X e^{-r^2_\vphi}dV_X}\right).
\end{eqnarray*}

Compared with the weak K\"{a}hler-Einstein case, it is expected that the regularity condition in the above definition is automatically satisfied. With this regularity assumption, on the regular part $X^{\rm reg}$, both sides of \eqref{eq-RFKC} are smooth volume forms and we have $r_{\vphi}\partial_{r_\vphi}=2 {\rm Re}(\xi_0)$ or, equivalently, $\xi_0=r_\vphi\partial_{r_\vphi}-i J(r_\vphi\partial_{r_\vphi})$.
Moreover, taking $\mathcal{L}_{r_\vphi \partial_{r_\vphi}}$ on both sides gives us the identity $\mathcal{L}_{r_\vphi\partial_{r_\vphi}}dV=2n\; dV$. 
Equivalently we have:
\[
\cL_{\xi_0}s= m n\cdot s,
\]
where $s\in |-mK_X|$ is the chosen $T$-equivariant non-vanishing holomorphic section. 
By Lemma \ref{lem-ldwt}, this implies $A_{X}(\wt_{\xi_0})=n$ (see \cite{HS17, LL16} for this identity in the quasi-regular case). The main result of \cite{MSY08} can be stated as follows.

\begin{thm}\label{thm-irSE}
If $(X, \xi_0)$ admits a Ricci-flat K\"{a}hler cone metric, then $A_X(\xi_0)=n$ and $\wt_{\xi_0}$ obtains the minimum of $\vol$ on $\frak{t}^{+}_\bR$.
\end{thm}

The following result partially generalizes Berman's result on K-polystability of K\"{a}hler-Einstein Fano varieties to the more general case of Ricci-flat Fano cones. Together with Theorem \ref{t-SDChigh},  it is used to show a generalization the minimization result \cite{MSY08}: the valuation $\wt_{\xi_0}$ minimizes $\hvol$ where $\xi_0$ is the Reeb vector field of the Ricci-flat Fano cone.

\begin{thm}[see \cite{CS15, LX16, LX17}] \label{thm-RF2K}
Assume $(X, \xi_0)$ admits a Ricci-flat K\"{a}hler cone metric. Then $A_{X}(\wt_{\xi_0})=n$ and $(X, \xi_0)$ is K-polystable among all special test configurations.
\end{thm}

\begin{proof}
Fix any smooth K\"{a}hler cone metric $\sddb r^2$ on $X$. Any special test configuration determines a geodesic ray $\{r_t^2=r^2 e^{\vphi_t}\}_{t> 0}$ of K\"{a}hler cone metrics. Denote $D(t)=D(\vphi_t)$.  
Then we have the following formula:
\begin{equation}
\lim_{t\rightarrow 0} \frac{D(t)}{-\log|t|^2}=\frac{D_{-\eta}\vol(\xi_0)}{\vol(\xi_0)}-(1-\lct(\cX, \cX_0))=D^\NA(\chi, \xi_0; \eta),
\end{equation}
which is a combination of two ingredients:
\begin{enumerate}
\item The Fano cone version of an identity from K\"{a}hler geometry which combined with \eqref{eq-Dvol} gives the formula: 
\begin{equation}
\lim_{t\rightarrow 0}\frac{E(\vphi_t)}{-\log|t|^2}=\frac{D_{-\eta}\vol(\xi_0)}{\vol(\xi_0)}.
\end{equation}
\item
$G(\vphi_t)$ is subharmonic in $t$ (cone version of Berndtsson's result) and its Lelong number at $t=0$ is given by $1-\lct(\cX, \cX_0)$ (cone version of Beman's result).
\end{enumerate}
The other key result is the cone version of Berndtsson's subharmonicity and uniqueness result, which was used to characterize the case of vanishing Futaki invariant. 
\end{proof}
\begin{rem}The argument in \cite{LX17} gives a slightly more general result: 
Assume $(X, \xi_0)$ admits a Ricci-flat K\"{a}ler cone metric, then $A_{X}(\wt_{\xi_0})=n$ and $(X, \xi_0)$ is Ding-polystable among $\bQ$-Gorenstein test configurations (see Remark \ref{r-Ding}).
\end{rem}

\section{Stable degeneration conjecture}\label{s-SDC}

In this section, we give a conjectural description of minimizers for general klt singularities, and explain various parts of the picture that we can establish. 

\subsection{Statement}
For a klt singularity $x\in (X,D)$, one main motivation to study the minimizer $v$ of $\hvol_{(X,D),x}$ is to establish a `local K-stability' theory, guided by the local-to-global philosophy mentioned in the introduction. In particular, we propose the following conjecture for all klt singularities. 
 \begin{conj}[Stable Degeneration Conjecture, \cite{Li15, LX17}]\label{conj-local}
 Given any arbitrary klt singularity $x\in (X={\rm Spec}(R), D)$, there is a  unique minimiser $v$ up to rescaling. Furthermore, $v$ is quasi-monomial, with a finitely generated associated graded ring $R_0=_{\rm defn}{\rm gr}_v(R)$, and the induced degeneration
 $$(X_0={\rm Spec}(R_0), D_0, \xi_v)$$ is a K-semistable Fano cone singularity. (See below for the definitions.)  \end{conj}

Let us explain the terminology in more details: First by the grading of $R_0$, there is a $T\cong \mathbb{C}^r$-action on $X_0$ where $r$ is the rational rank of $v$, i.e. the valuative semigroup $\Phi$ of $v$ generates a group $M\cong \mathbb{Z}^r$. Moreover, since the valuation $v$ identifies $M$ to a subgroup of $\mathbb{R}$ and sends $\Phi$ into $\mathbb{R}_{\ge 0}$, it induces an element in the Reeb cone $\xi_v$. If $R_0$ is finitely generated, then \cite{LX17} shows that we can embed $(x\in X)\subset (0\in \mathbb{C}^N)$ and find an rational vector $\xi \in \mathfrak{t}^{+}_{\mathbb{R}}\cap N_{\mathbb{Q}}$ sufficiently close to $\xi_v$ such that the $\mathbb{C}^*$-action generated by $\xi$ degenerates $X$ to $X_0$ with a good action. We denote by $o$ (or $o_{X_0}$) the unique fixed point on $X_0$. Furthermore, the extended Rees algebra yielding the degeneration does not depend on the choice of $\xi$. So we can define $D_0$ as the degeneration of $D$.

Conjecture \ref{conj-local}, if true, would characterize deep properties  of a klt singularity. Various parts are known, see Theorem \ref{t-high}. However, the entire picture remains open in general. 

\subsection{Cone case}\label{ss-cone}

The study of the case of cone is not merely verifying a special case. In fact, since  the stable degeneration conjecture predicts the degeneration of any klt singularities to a cone, understanding the cone case is a necessary step to attack the conjecture. Here we divide our presentations into two case: the rank one case and the general higher rank case. Although our argument in the higher rank case covers the rank one case with various simplifications, we believe it is easier for reader to first understand the rank one case, as it is equivalent to the more standard  K-semistability theory of the base which is a log Fano pair.  This connection is made via the theory of $\beta$-invariant, which is first introduced in \cite{Fuj18} in terms of ideal sheaves and further developed in \cite{Li17, Fuj16} via valuations. 

 \subsubsection{Rank one case} The rank one Fano cone is just a cone over a log Fano pair. More precisely, let $(S,B)$ be an $(n-1)$-dimensional log Fano pair, and $r$ a positive integer such that $r(K_S+B)$ is Cartier. Then we can consider the minimizing problem of the normalized volume at the vertex of the cone 
$$x\in (X,D)=C(S,B; -r(K_S+B)).$$ Such a question was first extensively studied in \cite{Li17}. More precisely, there is a canonical divisorial valuation obtained by blowing up $x$ to get a divisor $S_0$ isomorphic to $S$, which yields the degeneration of $x\in (X,D)$ to itself with $\xi$ being the natural rescaling vector field from the cone structure. Therefore, the stable degeneration conjecture predicts $v_{S_0}=\ord_{S_0}$ is a minimizer of $\hvol_{(X,D),x}$ if and only if $(S,B)$ is K-semistable, and this is confirmed in \cite{Li17, LL16, LX16}.
\begin{thm}\label{t-rankonecone}
The valuation $v_{S_0}$ is a stabilizer of $\hvol_{(X,D),x}$ if and only if $(S,B)$ is K-semistable. Moreover, $\hvol(S_0)< \hvol(E)$ for any other divisor $E$ over $x$.
\end{thm}
  In the below, we will sketch the ideas of two slightly different proofs of Theorem \ref{t-rankonecone}.

In the first approach, we carry out a straightforward calculation as follows: Given a compactified nontrivial special test configuration $(\mathcal{S},\mathcal{B})$ of $(S,B)$, then we obtain a valuation $v^*$ by restricting the divisorial valuation of the special fiber $S_0$ to $K(S)\subset K(S\times \mathbb{A}^1)$, which is a multiple of some divisorial valuation (cf. \cite{BHJ17}). Such a valuation $v^*$ pull backs a valuation $v^*_X$ on $K(X)$. Then we define a $\mathbb{C}^*$-valuation on $K(X)$ by
$v_{\infty}(f_m)=v^*_X(f_m)-mra_S(v^*)$ over $X$ for any $f_m\in H^0(S,-mr(K_S+B))$. In other words, $v_{\infty}=v^*_X-ra_S(v^*)v_{S_0}$, and we know that the induced filtration on $R$ yields the Duistermaat-Heckman (DH) measure of $(\mathcal{S},\mathcal{B})$ (see \cite[Definition 3.5]{BHJ17}). We define the ray in
$$\left\{v_t= v_{S_0}+t \cdot v_{\infty}\in \Val_{X,x} \ | \ t\in [0,\frac{1}{ra_S(v^*)}) \right\}.$$
Then the key computation in \cite{Li17} is that
\begin{eqnarray}
\frac{d}{dt}\hvol(v_t) |_{t=0}&=& \frac{n}{r^n}(-K_S-B)^{n-1}\cdot{\rm Fut}(\mathcal{S},\mathcal{B}).
\end{eqnarray}
In fact, if for any valuation $v$ over $S$,  we denote by $R_m=H^0(S,-mr(-K_S-B))$ and define 
$$\mathcal{F}^x_{v}R_m:=\{f\in R_m |\ \  f\in H^0(S, -mr(-K_S-B)\otimes \fa_x) \}, $$
then we easily see
$$\fa_{k}(v_t) \cap R_m=\mathcal{F}^{\frac{k-m}{t}}_{v_{\infty}}H^0(S, -mr(-K_S-B)). $$
So 
\begin{align} 
\vol(v_t)&= \lim_k \frac { l_{\mathbb{C}}(R/\fa_{k}(v_t))}{k^{n}/n!}  \nonumber \\
 & =\lim_{k\to \infty\ }\frac{n!}{k^n} \sum_{m=0}\left( \dim \mathcal{F}_{v_{\infty}}^0R_m-\dim\mathcal{F}^{\frac{k-m}{t}}_{v_{\infty}}H^0(S, -mr(-K_S-B))\right) \nonumber\\
 &=-\int^{\infty}_{-\infty} \frac{d\vol(\mathcal{F}_{v_{\infty}}R^{(x)})}{(1+t x)^{n}} \label{e-changevb},
\end{align}
where $\mathcal{F}_{v_{\infty}}R^{(x)}:=\bigoplus_{m} \mathcal{F}_{v_{\infty}}^{mx}R_m$ and the last equality is obtained by a change of variables (see Lemma \cite[Lemma 4.5]{Li17}). 

Since $A(S_0)=\frac{1}{r}$ and $A(v_{\infty})=0$, $A_{v_t}=\frac{1}{r}$, so 
$$\hvol(v_t)=-(\frac{1}{r})^n\int^{\infty}_{-\infty} \frac{d\vol(\mathcal{F}_{v_{\infty}}R^{(x)})}{(1+t x)^{n}},$$
and this implies that
\begin{eqnarray*}
\frac{d}{dt}\hvol(v_t) |_{t=0}&=&\frac{n}{r^n}\int^{\infty}_{-\infty} x\cdot d\vol(\mathcal{F}_{v_{\infty}}R^{(x)})\\
&=&\frac{n}{r^n} \lim_{k\to \infty}\frac{w_k}{kN_k}\\
&=&-\frac{1}{r^n} (-K_{\mathcal{S}}-\mathcal{B})^{n},\\
&=&\frac{n}{r^n} (-K_S-B)^{n-1}\cdot \Fut(\mathcal{S},\mathcal{B}).
\end{eqnarray*}
where for the second equality we use that $v_{\infty}$ is the DH measure for $(\mathcal{S},\mathcal{B})$. 

It is not straightforward to reverse the argument to show that $(S,B)$ is K-semistable implies that $\ord_{S_0}$ is a minimizer of $\hvol_{(X,D),x}$, since a priori there could be more complicated valuations than those induced by central fibres of test configurations. In particular, originally in \cite{Li17}, the techniques of `taking the limit of a sequence of filtered linear systems' developed in \cite{Fuj18} were used  in the case when the associated bigraded ring
 $$\bigoplus_{m,k}H^0(S, -rm(K_S+B)\otimes \fa_k)$$ is not finitely generated, and this is enough to treat all $\mathbb{C}^*$-equivariant valuations. 
 
 In \cite{LX16}, after the MMP method was systematically applied, it was shown that  
 \begin{eqnarray}\label{e-Ckollar}
 \inf_{v\in \Val_{X,x}} \hvol(v)=\{\inf \hvol(\ord_S)\  | \ \ \mbox{$\mathbb{C}^*$-equivariant Koll\'ar components $S$} \}
 \end{eqnarray}
 (see \eqref{e-kol} and the discussion below it). 
Since Koll\'ar components yield special degenerations,  therefore, the above arguments can be essentially reversed. See Section \ref{sss-highrankcone}. 
\begin{rem}\label{r-onetoone}In fact, we establish a one-to-one correspondence between special test configurations of $(S,B)$ (up to a base change) and  rays in $\Val_{X,x}$ emanating from $v_{S_0}$ containing a Koll\'ar component (different with $v_{S_0}$). 

An interesting consequence is that the above argument indeed gives an alternative way to show that K-semistability implies the valuative criterion of K-semistability with $\beta$-invariant as in \cite{Fuj16, Li17}, but without using the arguments of `taking a limit of filtered linear systems'.  
\end{rem}
\medskip
 
 The second approach to treat the cone singularity is developed in \cite{LL16} (see also \cite{LX16}). It is shown that K-semistablity of $(S,B)$ is equivalent to that of $(\bar{X},\bar{D}+(1-\frac{1}{rn})S_\infty)$, where $(\bar{X},\bar{D})$ is the projective cone of $(X,D)$ with respect to $-r(K_X+D)$ and $S_{\infty}(=S)$ is the divisor at the infinity place. This follows from a straightforward Futaki invariant calculation as in \cite[Proposition 5.3]{LX16}.  Applying the inequality \ref{thm:liuvolcomp} to $x\in (\bar{X},\bar{D}+(1-\frac{1}{rn})S_\infty)$, we immediately conclude that 
 \begin{eqnarray}\label{q-LL}
 \hvol(x, \bar{X},\bar{D})\ge\frac{(-K_S-B)^{n-1}}{r^n} =\hvol_{(X,D),x}(\ord_{S_0}). 
\end{eqnarray}
 

To understand better the relation between the K-semistability of $(\bar{X}, \bar{D}+(1-\frac{1}{rn})S_\infty)$ and of $(S,B)$, we want to present a direct calculation which connects the calculation on $\beta$-invariant on $(\bar{X},\bar{D}+\frac{1}{rn}S_\infty)$ and the one on $(S,B)$. 
  \begin{lem}\label{l-kkcone}
 Assume $\beta$-invariant is nonnegative for any divisorial valuation over $S$. Denote by $\hat{L}=\mathcal{O}(1)=\mathcal{O}(S_\infty)$ and $\delta=\frac{n+1}{rn}$. 
 For any $\bC^*$-invariant divisorial valuation $E$. 
We have the following
 \begin{equation}\label{eq-XDVstable}
\beta(E):=A_{(\bar{X},\bar{D}+(1-\frac{1}{rn})S_\infty)}(E)-\frac{\delta}{\hat{L}^n}\int_0^{+\infty} \vol(\cF_{\ord_{E}} \hat{R}^{(x)})dx\ge 0,
\end{equation}
where $\hat{R}=\bigoplus_{m=0}^{+\infty} H^0(\bar{X},m\hat L)$.
  \end{lem}
  The key of the proof is to relate the $\beta$-invariant for a $\bC^*$-invariant valuation $v$ over $\bar{X}$ to the $\beta$-invariant of the restriction of $v$ over the base $S$. 
\begin{proof}We have $K_{\bar{X}}+\bar{D}+(1-\frac{1}{rn})S_\infty=-\frac{n+1}{rn}\hat{L}=-\delta\hat{L}$, and define
$$\mathcal{F}^x_{v}{\hat{R}}_m:=\{f\in \hat{R}_m |\ \  f\in \hat R_m=\oplus_{0\le k \le m}H^0(S, kr(-K_S-B)) \mbox{  and } v(f)\ge x \}, $$

For any $\bC^*$-invariant divisorial valuation $v=\ord_E$ on $\bar{X}$, there exists $c_1\in \bZ$, $a\ge 0$ and a divisorial valuation $\ord_F$ over $S$ such that for any $f\in H^0(S, mr(-K_S-B))$, we have
\begin{eqnarray*}
v(t)=c_1; \text{ and } 
v(f)=a \cdot \ord_{F}(f)=:\bar{v}(f). 
\end{eqnarray*}
We estimate $\beta(E)$ in three cases depending on the signs of $a$ and $c_1$:

 $(a=0):$ The valuation $v$ is associated to the canonical $\bC^*$-action along the ruling of the cone, up to rescaling, then we easily get $\beta(E)=0$

 $(a>0$ and $c_1\ge 0):$ Then the center of $v$ is contained in $S_\infty$. In this case we can easily calculate:
\begin{eqnarray*}
\vol(\cF \hat{R}^{(x)})&=&\lim_{m\rightarrow +\infty}\frac{\dim_{\bC} \cF^{xm}\hat{R}_m}{m^n/n!}=\lim_{m\rightarrow+\infty} \frac{1}{m^n/n!} \sum_{k=0}^m \dim_{\bC} \cF^{xm-c_1(m-k)}_{\bar{v}} R_k\\
&=&n \int_0^1 \vol(\cF_{\bar{v}} R^{(c_1+\frac{x-c_1}{\tau})})\tau^{n-1} d\tau,
\end{eqnarray*}
where the last identity can be proved in the same way as in \eqref{e-changevb}. So we have:
\begin{eqnarray*}
\int_0^{+\infty}\vol(\cF \hat{R}^{(x)})dx&=&n\int_0^{+\infty}dx \int_0^1 \vol(\cF_{\bar{v}}R^{(c_1+\frac{x-c_1}{\tau})})\tau^{n-1}d\tau\\
&=&n\int_0^1 \tau^{n-1}d\tau \int_0^{+\infty} \vol(\cF_{\bar{v}}R^{(c_1+\frac{x-c_1}{\tau})})dx \\
&=&n\int_0^1 \tau^{n-1}d\tau\left[H^{n-1}c_1(1-\tau)+\tau\int_{c_1}^{+\infty}\vol\left(\cF_{\bar{v}}R^{(y)}\right)dy\right]\\
&=&\frac{c_1}{n+1}+n\int_0^{+\infty} \vol(\cF_{\bar{v}}R^{(x)})dx \int_0^1 \tau^n d\tau \\
&=&\frac{c_1}{n+1}+\frac{n}{n+1}\int_0^{+\infty}\vol(\cF_{\bar{v}}R^{(x)})dx.
\end{eqnarray*}

On the other hand, we have $H^{n-1}=\hat{L}^n$ and: 
$$A_{(\bar{X},\bar{D}+(1-\beta)S_\infty)}(\ord_E)=A_{(S,B)}(\bar{v})+c_1-(1-\beta)c_1=A_{(S,B)}(\bar{v})+\frac{c_1 }{rn}$$ 
So we get:
\begin{eqnarray*}
\beta(E)&=&A_{(S,B)}(\bar{v})+\frac{c_1}{rn}-\frac{\frac{n+1}{rn}}{H^{n-1}}\frac{n}{n+1}\left(\frac{c_1}{n+1}+ \int_0^{+\infty}\vol(\cF_{\bar{v}}R^{(x)})dx\right)\\
&= & A_{(S,B)}(\bar{v})-\frac{1}{rH^{n-1}}\int_0^{+\infty}\vol(\cF_{\bar{v}}R^{(x)})dx=\beta(\bar{v}),
\end{eqnarray*}
which is non-negative by our assumption. 

$(a>0$ and $c_1<0)$: In this case, the center of $v$ is at the vertex. As a consequence we have:
\begin{eqnarray*}
A_{(\bar{X}, \bar{D}+(1-\beta)S_\infty)}(v)&=&A_{(S,B)}(\bar{v})+(-c_1)+(\frac{1}{r}-1) (-c_1) \\
&=&A_{(S,B)}(\bar{v})+\frac{-c_1}{r}\ge A_{(S,B)}(\bar{v}).
\end{eqnarray*}
The similar calculation as in the second case shows that $\beta(E)\ge \beta(\bar{v})$. 
\end{proof}

\bigskip Finally, to show $\hvol(S_0)<\hvol(E)$ for $E\neq S_0$, in \cite{LX16}, it was first proved that if $E$ is a minimizer then it has to be a $\mathbb{C}^*$-equivariant Koll\'ar component. Then a careful study of the geometry of $E$ using the equality condition in \eqref{q-LL} implies $E=S$.  This is similar to the analysis for the equality case in \cite{Fuj18, Liu16} where they showed that the K-stable $\mathbb{Q}$-Fano variety with the maximal volume $(n+1)^n$ can only be $\mathbb{CP}^n$. We will leave the discussion on this uniqueness type result to the general case of cones of higher rational ranks, where we take a somewhat different approach, using more convex geometry.  

\begin{rem}It is worthy pointing out that there is another global invariant for an $n$-dimensional log Fano pair $(S,B)$, defined as 
$$\delta(S,B)=\inf_{v\in \Val_S} \frac{A_{(S,B)}(v)\cdot(-K_S-B)^{n}}{\int^{\infty}_{0}\vol(-K_S-B-tv)dt}$$
 (see \cite{FO16, BJ17}). $\delta$-invariant shares lots of common properties with the normalized volume. For example, the existence of minimizers were proved using similar strategy. 
They both have differential geometric meanings. The minimizer of $\hvol$ is related to the metric tangent cone (see section \ref{ss-dsc}); while the valuation on $K(S)$ yielding $\delta(S,B)$ is related to the the existence of twisted K\"ahler-Einstein metrics (see \cite{BoJ18}).  

For a log Fano pair $(S,B)$ and a cone $x\in (X,D)=C(S,B;-r(K_S+B))$, if $(S,B)$ is not K-semistable, or equivalently $\delta=\delta(S,B)<1$, then we have 
$$\hvol(x, X, D)\ge \frac{\delta^n\cdot (-K_S-B)^{n-1}}{r^n}.$$ 
This follows from our second proof  by looking at $(\bar{X}, \bar{D}+(1-\beta)S_\infty)$ and applying the inequality \cite[Theorem D]{BJ17} which can be written as
\[
(K_{\bar{X}}+\bar{D}+(1-\beta)S_\infty)^{n}\le\frac{(n+1)^n}{n^n}\cdot \hvol(x, X, D)\cdot \bar{\delta}^n,
\]
where $\bar{\delta}:=\delta(\bar{X},\bar{D}+(1-\beta)S_\infty)$. We claim $\min\{\bar{\delta},1\}=\delta$. In fact, by the argument in \cite[Section 7]{BJ17}, we know that $\bar{\delta}$ is computed by a $\mathbb{C}^*$-invariant valuation and the claim follows from the calculation in the proof of Lemma \ref{l-kkcone}.
\end{rem}

\subsubsection{Log Fano cone in general}\label{sss-highrankcone}
We proceed to investigate a log Fano cone $o\in (X, D,\xi)$ where the torus $T$ could have dimension larger than one. However, we consider not only the valuations in $\mathfrak{t}^+_{\mathbb{R}}(X)$ coming from the torus as in \cite{MSY08} (see Section \ref{s-tvariety}) but all valuations in $\Val_{X,o}$. Compared to the proof of Theorem \ref{t-rankonecone}, for the higher rational rank case, we rely more on the construction of Koll\'ar components coming from the birational geometry. More explicitly, we use the relation between special test configurations and Koll\'{a}r components (see \cite[2.3]{LX16} and \cite[3.1]{LX17}).

By  the results from the MMP (see \eqref{e-kol} and the explanation below),  to show a valuation is a minimizer in $\Val_{X,x}$, we only need to show its normalized volume is not greater than that of any $T$-invariant Koll\'ar component. On the other hand, any T-equivariant Koll\'ar component $E$ in $\Val_{X,o}$ yields a special  test configuration of $(\mathcal{X},\mathcal{D},\xi;\eta)$ of $(X,D)$ such that $-\eta\in \ft^+_\bR(X_0)$ and the valuation associated to $-\eta$ coincides with $\ord_E$. We denote by $(X_0,D_0)$ the fiber with a cone vertex $o$. Then we can compare the volumes as $\hvol_{X}(\xi)=\hvol_{X_0}(\xi)$ and $\hvol_{X}(E)=\hvol(-\eta)$. Since $\xi,-\eta\in \ft^+_{\bR}(X_0)$ we reduce the question to the set up of \cite{MSY08} on $X_0$. Then we only need to each time treat one degeneration $X_0$ and try to understand how to pass properties between $X_0$ and $X$. 

With this strategy, we can show the following generalization of Theorem \ref{t-rankonecone}.  
\begin{thm}[\cite{LX17}]\label{t-SDChigh}
Let $x\in (X, D,\xi)$ be a log Fano cone singularity. Then $v_{\xi}$ is a minimizer of $\hvol_{(X, D),x}$ if and only if  $(X, D,\xi)$ is K-semistable. In such case, 
$\hvol(v_{\xi})<\hvol(v)$ for any quasi-monomial valuation $v$ if $v$ is not a rescaling of $v_{\xi}$.
\end{thm}

If $(X, D,\xi)$ is K-semistable, then for each special test configuration $(\mathcal{X},\mathcal{D},\xi;\eta)$, on $X_0$, we can consider the ray $\xi_t=\xi-t\eta$ for $t\in [0, \infty)$. We know 
$$\frac{d}{dt}\hvol_{(X_0,D_0),o}(v_{\xi_t})|_{t=0}=c\cdot {\rm Fut}(\mathcal{X},\mathcal{D},\xi;\eta)\ge 0.$$
Moreover, when $(X_0,D_0,o)=(X_0, \emptyset, o)$ is an isolated singularity, it was shown in \cite{MSY08} that  $\hvol(v_{\xi_t})$ is a convex function. We obtain a stronger result for any log Fano cone $(X_0,D_0,\xi_0)$ (see Section \ref{sss-convex}). In particular, we conclude that $\hvol(v_{\xi_t})$ is an increasing function of $t$, and its limit is $\hvol(-\eta)$, thus the inequality in the following relation holds true:
$$\hvol_{(X, D),x}(\xi)=\hvol_{(X_0, D_0),o}(\xi)\le \hvol_{(X_0, D_0),o}(-\eta)=\hvol_{(X,D),x}(E).$$

The first identity consists of two identities: $A_{(X,D)}(v_\xi)=A_{(X_0, D_0)}(v_\xi)$ and $\vol_{X}(v_\xi)=\vol_{X_0}(v_\xi)$, which essentially follow from the flatness of $T$-equivariant test configuration (see \cite[Lemma 3.2]{LX17}).  The last identity is because $v_{-\eta}=\ord_E$.

This argument is  reversible since we can indeed attach to any special test configuration such a set of valuations (see Remark \ref{r-onetoone}): if we consider the valuation $w_{t}$ obtained by considering the vector field $\xi_t$ as a valuation on $K(\cX)$ and then take its restriction on $K(X)$. The corresponding degeneration induces the test configuration. See \cite[6]{LX16} and \cite[4.2]{LX17} for more details.

\subsubsection{Uniqueness}\label{sss-convex} 
We have seen the convexity of the normalized volume function in the Reeb cone  plays a key role.  In \cite{MSY08}, the strict convexity on the normalized function is established for the valuation varying inside the Reeb cone for an isolated singularity. 
This is the kind of property we need for the uniqueness of the minimizer of a K-semistable Fano cone singularity $(X,D,\xi)$.  However, as we do not know the associated graded ring of other minimizer is finitely generated, we can not degenerate two minimizers into the Reeb cone. Thus  we need develop a technique to deal with valuations outside the Reeb cone.

The idea of the argument in \cite[Section 3.2]{LX17} is to use the theory of Newton-Okounkov bodies which was first developed in \cite{LM09, KK12}) and in the local setting in \cite{Cut13, KK14}. This is a theory which realizes the volumes in algebraic geometry with an asymptotic nature to the Euclidean volumes of some convex bodies in $\mathbb{R}^n$.  So our aim is to apply the Newton-Okounkov body construction to translate the normalized volume of valuations into the volume of convex bodies, and then invoke a convexity property of the volumes functions known in the latter setting. 

To start, we first need to set a valuation $\mathbb{V}$ with $\mathbb{Z}^n$-valued valuation, which sends the elements in $R$ to the lattice points inside a convex region $\tilde{\sigma}$, so that later we can realize the normalized volumes of valuations as the volume of subsets in $\tilde{\sigma}$.

For any fixed $T\cong (\bC^*)^r$-equivariant quasi-monomial valuation $\mu$, we know it is of the form $( \xi_{\mu},v^{(0)})$ where $\xi_{\mu}\in M_{\mathbb{R}}$ and $v^{(0)}$ is a quasi-monomial valuation over $K(Y)$, such that for any function $f \in R_u$, 
$$\mu(f)=\langle \xi_{\mu},u\rangle+v^{(0)}(f)$$
(see Theorem \ref{t-Tcano}(1)). 
We fix a lexicographic order on $\bZ^{r}$ and define for any $f\in R$, 
\[
\bV_1(f)= \min\{u; f=\sum_u f_u \text{ with } f_u \neq 0\}=\bV_1(f),
\]
i.e., the first factor $\bV_1$ comes from the toric part of $\mu$. 

We extend this $\bZ^{r}$-valuation $\bV_1$ to become a $\bZ^n$-valued valuation in the following way: Denote $u_f=\bV_1(f)\in \sigma^{\vee}$ and $f_{u_f}$ the corresponding nonzero component. 
Define $\bV_2(f)=v^{(0)}(f_{u_f})$. Because $\{\beta_i\}$ are $\mathbb{Q}$-linearly independent, we can write
$\bV_2(f)=\sum_{i=1}^{s} m^{*}_i \beta_i$ for a uniquely determined $m^*:=m^*(f_{u_f})=\{m^*_i:=m^*_i(f_{u_f})\}$. Moreover, the Laurent expansion of $f$ has the form:
\begin{equation}\label{eq-flaurent*}
f_{u_f}=z_1^{m^*_1}\dots z_{s}^{m^*_{s}} \chi_{m^*}(z'')+\sum_{m\neq m^*} z_1^{m_1}\dots z_{s}^{m_{s}} \chi_m(z'').
\end{equation}
 Then $\chi_{m^*}(z'')$ in the expansion of \eqref{eq-flaurent*} is contained in $\bC(Z)$, where on some model of $Y$, we have $Z=\{z_1=0\}\cap \dots \{z_s=0\}=D_1\cap \dots\cap D_{s}$ is the center of $v^{(0)}$. 

Extend the set $\{\beta_1, \dots, \beta_{s}\}$ to $d=n-r$ $\bQ$-linearly independent positive real numbers $\{\beta_1, \dots, \beta_{s}; \gamma_1, \dots, \gamma_{d-s}\}$. 
Define
$\bV_3(f)=w_{\gamma}(\chi_{m^*}(z''))$ where $w_{\gamma}$ is the quasi-monomial valuation with respect to the coordinates $z''$ and the $(d-s)$ tuple $\{\beta_1, \dots, \beta_{s}; \gamma_1, \dots, \gamma_{d-s}\}$.

Now we assign the lexicographic order on 
$$\mathbb{G}:=\bZ^{r}\times G_2\times G_3\cong \bZ^{r}\times \bZ^{s}\times \bZ^{n-r-s}$$ and define $\mathbb{G}$-valued valuation:
\begin{equation}
\bV(f)=(\bV_1(f), \bV_2(f_{u_f}), \bV_3(\chi_{m^*})).
\end{equation}

Let $\mathcal{S}$ be the valuative semigroup of $\bV$. Then $\mathcal{S}$ generates a cone $\tilde{\sigma}$ which is the one we are looking for. We also let $P_1: \bR^n\rightarrow \bR^{r}$, $P_2: \bR^n\rightarrow \bR^{s}$ and $P=(P_1, P_2): \bR^n\rightarrow \bR^{r+s}$ be the natural projections. Then $P_1(\tilde{\sigma})=\sigma\subset \bR^r$.

\medskip

To continue, we consider how to construct some subsets $\Delta_{\tilde{\Xi}_t} \subset \tilde{\sigma}$ whose Euclidean volume is the same as the normalized volumes of the valuations.
For any $\xi\in {\rm int}(\sigma)$, denote by $\wt_\xi$ the valuation associated to $\xi$. We can connect $\wt_\xi$ and $\mu$ by a family of quasi-monomial valuations: $\mu_t= ((1-t)\xi+t \xi_{\mu}, t v^{(0)})$ defined as
\[
\mu_t(f)=t v^{(0)}(f)+\langle u, (1-t)\xi+t\xi_{\mu}\rangle \mbox{\ \ \ for any $f\in R_u$}.
\]
So the vertical part of $\mu_t$ corresponds to the vector $\Xi_t:=((1-t)\xi+t\xi_{\mu}, t\beta)\in \bR^{r+s}$. Extend $\Xi_t$ to 
$\tilde{\Xi}_t:=(\Xi_t, 0)\in \bR^n$ and define the following set:
\[
\Delta_{\tilde{\Xi}_t}=\left\{y\in \tilde{\sigma}; \langle y, \tilde{\Xi}_t\rangle \le 1 \right\}=\left\{y\in \tilde{\sigma}; \langle P(y), \Xi_t\rangle \le 1\right\}.
\]
Because $\hvol$ is rescaling invariant, we can assume $A_{(X,D)}(v)=A_{(X,D)}(\xi)=1$. Then by the $T$-invariance of $v_t$, we easily get:
$$A(v_t)=t A(v^{(0)})+A_{(X,D)}((1-t)\xi+t\zeta)=t A_{(X,D)}(v)+(1-t)A_{(X,D)}(\xi)\equiv 1.$$
The Newton-Okounkov body theory implies that we have
\[
\hvol(v_t)=\vol(v_t)=\vol(\Delta_{\tilde{\Xi}_t}).
\]

\medskip

To finish the uniqueness argument, now we only need to look at the convex geometry of $\Delta_{\tilde{\Xi}_t}$.
We note that $\tilde{\Xi}_t$ is linear with respect to $t$, and each region $\Delta_{\tilde{\Xi}_t}$ is cut out by a hyperplane $H_t$ on the convex cone $\tilde{\sigma}$. Moreover, all $H_t$ passes through a fixed point. A key result from convex geometry then shows that  $\phi(t):=\vol(\Delta_{\tilde{\Xi}_t})$ is strictly convex as a function of $t\in [0,1]$ (see \cite{MSY06, Gig78}). By the assumption $\phi(0)=\vol(v_0)=\hvol(\wt_{\xi})$ is a minimum. So the strict convexity implies 
$$\phi(1)=\vol(\Delta_{\tilde{\Xi}_1})=\hvol(v)>\hvol(\wt_\xi)=\phi(0).$$

\subsection{Results on the general case}\label{ss-general}

To treat the general case, the key idea, suggested by the degeneration conjecture, is to understand how  an arbitrary klt singularity can be degenerated to a K-semistable Fano cone singularity. In \cite{LX16}, by localizing the setting of \cite{LX14}, the following approach of using Koll\'ar components is developed.  

From each ideal $\fa$, we can take a dlt modification of 
$$f\colon (Y,D_Y)\to (X,D+\lct(X,D;\fa)\cdot\fa),$$ where $D_Y=f_*^{-1}D+{\rm Ex}(f)$ and for any component $E_i\subset {\rm Ex}(f)$ we have $$A_{X,D}(E)=\lct(X,D;\fa)\cdot \mult_{E}f^*\fa.$$ 
There is a natural inclusion $\D(D_Y)\subset \Val^{=1}_{X,x}$, and using a similar argument as in \cite{LX14}, we can show that there exists a Koll\'ar component $S$ whose rescaling in $\Val^{=1}_{X,x}$ contained in $\D(D_Y)$ satisfies that 
$$\hvol(\ord_S)= \vol^{\rm loc}(-A_{X,D}(S)\cdot S)\le \vol^{\rm loc}(-K_Y-D_Y)\le \mult(\fa)\cdot \lct^n(X,D;\fa).$$
Here $\vol^{\rm vol}(\cdot)$ is the local volume of divisors over $X$ as defined in \cite{Ful13}. Then Theorem \ref{thm:liueq} immediately implies that
\begin{eqnarray}\label{e-kol}
\hvol(x, X,D)=\inf\{\hvol(\ord_S)|\  S \mbox{ is a Koll\'ar component over $x$}\}.
\end{eqnarray}
Moreover, if $x\in (X,D)$ admits a torus group $T$-action, then by degenerating to the initial ideals, as the colengths are preserved and the log canonical thresholds  may only decrease, the infimum of the normalized multiplicities in Theorem \ref{thm:liueq} can be only run over all $T$-equivariant ideals. Then the equivariant MMP allows us to make all the above data $Y$ and $S$ be $T$-equivariant.

In case a minimizer is divisorial,  then the above discussion shows that
\begin{lem}[{\cite{LX16, Blu18}}]\label{l-dmkollar}  A divisorial minimizer of $\hvol_{X,D}$ yields a  Koll\'ar component. 
\end{lem}
 In general, we know that the minimizer is a limit of a rescaling of Koll\'ar components (see \cite{LX16}). So understanding the limiting process is crucial. When the minimizer is quasi-monomial $v$ of rational rank $r$, i.e., the valuation $v$ is \' etale locally a monomial valuation with respect to a log resolution $(Y,E)\to X$, then  a natural candidates will be the valuations given by taking rational approximations of the monomial coordinates $\alpha \in \mathbb{R}^r_{>0}$.
 
Our first observation in \cite{LX17} is  using  MMP results including the ACC of log canonical thresholds, we could construct a weak log canonical model which extracts divisors whose coordinates are good linear Diophantine approximations of the coordinates of $v$. 
\begin{prop}\label{p-highmodel}For any quasi-monomial valuation $v$ computing a log canonical threshold of a graded sequence of ideals, we can find a sequence of divisors $S_1$,..., $S_r$,
such that  
\begin{enumerate}
\item there is a model $Y\to X$ which precisely extracts $S_1$,..., $S_r$ over $x$,
\item  there exists a component $Z$ of $\cap^r_{i=1} S_i$ such that $(Y, E:=\sum^r_{i=1} S_i)$ is toroidal around the generic point $\eta(Z)$,
\item $v$ is \'etale locally a monomial valuation over  $\eta(Z)$ with respect to $(Y,E)$ (see Section \ref{ss-firstdefinition}),
\item $(Y,E)$ is log canonical, and $-K_Y-E$ is nef.
\end{enumerate}
\end{prop}

Fix the first model $Y_0=Y$, then one can construct a sequence of models $(Y_j, E_j)$ satisfying Proposition \ref{p-highmodel} such that a suitable rescaling of the components of $E_j$ become closer and closer to $v$. To make the notation easier, we rescale $v$ into $\Val_{X,x}^{=1}$. Similarly, we can embed the dual complex of a dlt modification of $(Y_j, E_j)$ into $\Val_{X,x}^{=1}$  (see \cite{dFKX}). Our construction moreover satisfies that
$$\DR(Y_0,E_0)\supset \DR(Y_1,E_1)\supset \cdots $$
 Then the above discussion indeed implies that

\begin{lem}A quasi-monomial minimizer $v\in \Val_{X,x}^{=1}$ can be written as a limit of $c_j\cdot \ord_{S_j}\in \DR(Y_j, E_j)$ where $S_i$ are Koll\'ar components. 
\end{lem}

It would be natural to expect that $c_j\cdot \ord_{S_j}$ is indeed contained in the simplex $\sigma_{\eta(Z)}\subset \Val_{X,x}^{=1}$ which corresponds to all the monomial valuations in $\Val_{X,x}^{=1}$ over $\eta(Z)$ with respect to $(Y,E)$. However, for now we can not show it. 

If we further assume $R_0=\gr_v(R)$ is finitely generated, then we have the following
\begin{prop}If $R_0=\gr_v(R)$ is finitely generated, then $\gr_{v}(R)\cong \gr_{v_i}(R)$ for any $v_i\in \sigma_{\eta(Z)}$ sufficiently close to $v$. 
\end{prop}

This immediately implies that $(X_0:={\rm Spec}(R_0),D_0)$ is semi-log-canonical (slc). The final ingredient we need is the following,
\begin{prop}Under the above assumptions on $(X,D)$ and its quasi-monomial minimizer $v$, then $\xi_v$ is a minimizer of $(X_0,D_0)$. In particular, 
$$\hvol(x, X,D)=\hvol(o, X_0,D_0).$$
\end{prop}
\begin{proof} We claim that $\xi_v$ is indeed a minimizer of $\hvol_{X_0,D_0}$. If not, we can find a degeneration $(Y,D_Y, \xi_{Y})$ induced by an irreducible anti-ample divisor $E$ over $o'\in X_0$ with 
$$\hvol_Y(\xi_E)=\hvol_{X_0}(\ord_{E})<\hvol_{X_0}(\xi_v)=\hvol_Y(\xi_Y).$$ 
This is clear by our discussion when $(X_0,D_0)$ is klt. The same thing still holds when the model extracting $S_j$ is only log canonical but not plt, which implies that $(X_0,D_0)$ is semi-log-canonical but not klt. In fact, denote by $(X^{\rm n}_0,D^{\rm n}_0)\to (X_0,D_0)$ the normalization, then Lemma \ref{l-volslc} implies that 
$$\hvol(o', X_0,D_0):=\sum_{o_i\to  o'} \hvol(o_i, X^{\rm n}_0,D^{\rm n}_0)=0$$ in this case. The argument in \cite[Lemma 4.13]{LX17} then says in this case, we can still extract an equivariant anti-ample irreducible divisor $E$ over $o'\in X_0$ with $\hvol(\ord_E)$ arbitrarily small. 

Then Lemma \ref{l-doudeg} shows that we can construct a degeneration from $(X,D)$ to $(Y,D_Y)$ and a family of valuations $v_t\in \Val_{X,x}$ for $t\in [0,\epsilon]$ (for some $0<\epsilon\ll 1$), with the property that
$$\hvol_X(v_t)=\hvol_Y(\xi_Y-t \eta)<\hvol_Y(\xi_Y)=\hvol_{X_0}(\xi_v)=\hvol_X(v),$$
where for the second inequality, we use again the fact that $\hvol_Y(\xi_Y-t\cdot\eta)$ is a convex function in this setting as well. But this is a contradiction.
\end{proof}
\begin{lem}\label{l-doudeg}
Let $(x\in X)\subset (0\in \mathbb C^N)$ be a closed affine variety. If $\lambda_1 \in \mathbb{N}^N$ is a coweight of $(\mathbb{C}^*)^N$ which gives an action degenerating $X$ to $X_0$ when $t\rightarrow 0$, and $\lambda_2  \in \mathbb{N}^N$ degenerates $X_0$ to $Y$ when $t\to 0$, then for $k\in \mathbb{N}$ sufficiently large, $k\lambda_1+\lambda_2$ degenerates $X$ to $Y_0$. 
\end{lem}
The proof was essentially given in \cite[section 6]{LX16} (see also \cite[Lemma 3.1]{LWX18}) and uses some argument in the study of toric degenerations (see e.g. \cite[Section 5]{And13}).

\begin{lem}\label{l-volslc}
If $o\in (X,D)$ is an lc but not klt point, then 
$$\hvol(o, X,D):=\inf_{v\in \Val_{X,o}}\hvol(v)=0.$$ 
\end{lem}
\begin{proof}

Let $\pi^{\rm dlt}\colon (X^{\rm dlt},D^{\rm dlt})\to (X,D)$ be a dlt modfication and pick $o''$ a preimage of $o$ under $\pi^{\rm dlt}$, then 
$\hvol(o'', X^{\rm dlt},D^{\rm dlt})\ge \hvol(o, X,D)$, thus we can assume $(X^{\rm dlt},D^{\rm dlt})$ is dlt $\bQ$-factorial. 

By specialing a sequence of points, and applying Theorem \ref{t-lower}, we can assume $o\in (X,D)$ is a point on a smooth variety with a smooth reduced divisor $D$.  
Now we can take a weighted blow up of $(1,\epsilon,....,\epsilon)$ where the first coordinate yields $D$. Then the exceptional divisor $E$ has its normalized volume
$$\hvol(E)=\frac{(n-1)^n\epsilon^n}{\epsilon^{n-1}}=(n-1)^n\epsilon\to 0 \mbox {\ as \  } \epsilon\to 0.$$
\end{proof}

This implies that $(X_0,D_0)$ is klt and $(X_0,D_0,\xi_v)$ is a K-semistable Fano cone.  To summarize, we have shown Part (a) in the following theorem which characterize what we know about the Stable Degeneration Conjecture \ref{conj-local} for a general klt singularity.

  \begin{thm}[{\cite[Theorem 1.1]{LX17}}]\label{t-high}
Let $x\in (X,D)$ be a klt singularity.  Let $v$ be a quasi-monomial valuation in $\Val_{X,x}$ that minimises $\hvol_{(X,D)}$ and has a  finitely generated associated graded ring  ${\rm gr}_v(R)$ (which is always true if the rational rank of $v$ is one by Lemma \ref{l-dmkollar}). Then the following properties hold:
\begin{enumerate}
\item[(a)] The degeneration
$\big(X_0=_{\rm defn}{\rm Spec}\big({\rm gr}_v(R)\big), D_0, \xi_v \big)$ is a K-semistable Fano cone, i.e. $v$ is a K-semistable valuation;
\item[(b)] Let $v'$ be another quasi-monomial valuation in $\Val_{X,x}$ that minimises $\hvol_{(X,D)}$. Then $v'$ is a rescaling of $v$.
\end{enumerate}
Conversely, any quasi-monomial valuation that satisfies (a) above is a minimiser. 
 \end{thm}

\begin{proof} 

We first show the uniqueness in general, under the assumption that it admits a degeneration $(X_0,D_0,\xi_v)$ given by a K-semistable minimiser $v$. For another quasi-monomial minimiser $v'$ of rank $r'$, by a combination of the Diophantine approximation and an MMP construction including the application of ACC of log canonical thresholds (see Proposition \ref{p-highmodel}),  we can obtain a model $f\colon Z\to X$ which extracts $r'$ divisors $E_i$ ($i=1,...,r'$) such that $(Z, D_Z=_{\rm defn}\sum E_i+f_*^{-1}D)$ is log canonical. Moreover,  the quasi-monomial valuation $v'$ can be computed at the generic point of a component of the intersection of $E_i$, along which $(Z,D_Z)$ is toroidal. 
Then with the help of the MMP, one can show  $Z\to X$ degenerates to a birational morphism $Z_0\to X_0$. Moreover, there exists a quasi-monomial valuation  $w$ computed on $Y_0$ which can be considered as a degeneration of $v'$ with 
$$\hvol_{X_0}(w)=\hvol_X(v')=\hvol_X(v)=\hvol_{X_0}(\xi_v).$$ Thus $w=\xi_v$ by Section \ref{sss-convex}  after a rescaling. Since $w({\bf in}(f))\ge v'(f)$ and $\vol(w)=\vol(v')$, we may argue this implies 
$$\xi_v({\bf in}(f))=v'(f)$$
(see \cite[Section 4.3]{LX17}). Therefore, $v'$ is uniquely determined by $\xi_v$. 

\medskip

To show the last statement, we  already know it for a cone singularity. For  a valuation $v$ on a general singularity $X$ such that the degeneration $(X_0,D_0,\xi_v)$ is K-semistable,  since the degeneration to the initial ideal argument implies that $\hvol(x, X,D)\ge \hvol(o, X_0,D_0)$,  then 
$$\hvol_X(v)=\hvol_{X_0}(\xi_v)=\hvol(o, X_0,D_0)$$ is equal to $\hvol(x, X,D)$. 
\end{proof}

So in other words, the stable degeneration conjecture precisely predicts the following two sets coincide:
\[ 
\left \{
  \begin{tabular}{c}
 Minimizers of  $\hvol$  
  \end{tabular}
\right \} \longleftrightarrow
 \Big\{
  \begin{tabular}{c}
K-semistable valuations  
  \end{tabular}
\Big\}.
\]
Theorem \ref{t-existence} and Theorem \ref{t-high} together imply the existence of left hand side and the uniqueness of the right hand side, as well as the direction that any K-semistable valuation is a minimizer.

\bigskip

 Finally, let us conclude this section with the two dimensional case.  
 \begin{thm}\label{thm-2dim}
Let $(X, D, x)$ be a two-dimensional log terminal singularity. The Stable Degeneration Conjecture \ref{conj-local} holds for $(X, D)$. Moreover, if $D$ is a $\bQ$-divisor, then the minimizer of $\hvol_{(X, D)}$ is always divisorial.
\end{thm}
\begin{proof}

We first consider the case when $X=\bC^2$. Let $v_*$ be a minimizer and denote $\frak{a}_\bullet=\{\frak{a}_m(v_*)\}_{m\in \bN}$. Then it was known that $v_*$ computes the log canonical threshold of $(X, D+\frak{a}_\bullet)$. By similar argument as \cite[]{JM12}, we know that $v_*$ must be quasi-monomial.

If $v_*$ is divisorial, then we know that the associated divisor is a Koll\'{a}r component. Otherwise, $v_*$ satisfies
${\rm rat.rk.}(v_*)=2$ and ${\rm tr.deg.}(v_*)=0$. From the description of valuations on $\bC^2$ using {\it sequences of key polynomials} (SKP), it was showed that the valuative semigroup $\Gamma$ of $v_*$ is finitely generated (see \cite[Theorem 2.28]{FJ04}). Since the residual field of $v_*$ is $\bC$, we know that ${\rm gr}_{v_*}R\cong \bC[\Gamma]$, which is finitely generated. By \cite{LX17}, we know that $v_*$ is indeed the unique minimizer of $\hvol$ (up to scaling) which is a K-semistable valuation. 

If $D$ is a $\bQ$-divisor and $v_*$ is not divisorial, then the pair $(X_0, D_0)$ is a $\bQ$-Gorenstein toric pair with $\bQ$-boundary toric divisor and the associated Reeb vector field $\xi_{v_*}$ solves the convex geometric problem. But in dimension 2 case (i.e. on the plane), it is easy to see that the corresponding convex geometric problem as discussed in section \ref{sss-convex} for toric valuations always has a rational solution. This is a contradiction to $v_*$ being non-divisorial.

More generally, we know that $X=\bC^2/G$ where $G$ is a finite group acting on $\bC^2$ without pseudo-reflections. Consider the covering $(\bC^2, \tilde{D}, 0)\rightarrow (X, D, x)$. Then by the above discussion, there exists a unique minimizer $v_*$ of $\hvol_{(\bC^2, \tilde{D}, 0)}$. In particular, $v_*$ is invariant under the $G$-action. So it descends to a minimizer of $\hvol_{(X, D, x)}$ which is quasi-monomial and has a finitely generated associated graded ring.

\end{proof}

\section{Applications}\label{s-application}

In this section, we give some applications of the normalized volume. We have seen that the normalized volume question of a cone singularity is closely related the K-semistability of the base. Another situation where singularities naturally appear is on the limit of smooth Fano manifolds.

\subsection {Equivariant K-semistability of Fano}\label{ss-equiva}

An interesting application of the minimizing theory is to treat the equivariant K-semistability. 
\begin{defn} A log Fano pair $(S,B)$ with a $G$-action is called $G$-equivariant K-semistable, if for any $G$-equivariant test configuration $ (\mathcal{S},\mathcal{B})$, the generalized Futaki invariant ${\rm Fut}(\mathcal{S},\mathcal{B})\ge 0$.  We can similarly define $G$-equivariant K-polystability. 
\end{defn}
The notion of usual K-(semi,poly)stability trivially implies the equivariant one. It is a natural question to ask whether they are equivalent, and if it is confirmed it will reduce the problem of verifying K-stability into a much simpler ones if the log Fano pair carries a large symmetry. When $S$ is smooth and $B=0$, this is proved in \cite{DS16}, using an analytic argument. Here we want to explain how our approach can  give a proof of such an equivalence when $G=T$ is a torus group.  

The key is  the fact we obtain in \eqref{e-Ckollar} and \eqref{e-kol} : let $x\in (X,D)$ be a klt singularity which admits a $T$ action for a torus group $T$, then 
 \begin{eqnarray}\label{e-Tkollar}
 \ \ \ \inf_{v\in \Val_{X,x}} \hvol(v)=\{\inf(\hvol(\ord_S)) | \ \ \mbox{$T$-equivariant Koll\'ar components $S$} \}.
 \end{eqnarray}

So if $(S,B)$ is not K-semistable, by Theorem \ref{t-rankonecone}, we know that over the cone $x\in (X,D)$, the valuation $\ord_{S_{\infty}}$ obtained by the canonical blow up does not give a minimizer. By  \eqref{e-Tkollar}, there exists a $T$-equivariant valuation $v$ such that $\hvol(v)<\hvol(\ord_{S_{\infty}}) $.  So we can find a $T$-equivariant Koll\'ar component $S$ such that $\hvol(\ord_S)<\hvol(\ord_{S_{\infty}})$. Then arguing as before, we can find a $T$-equivariant test configuration $(\mathcal{S},\mathcal{B})$ with ${\rm Fut}(\mathcal{S},\mathcal{B})< 0$. 

To prove a similar statement for K-polystability is more delicate. Assume a K-semistable log pair $(S,B)$ admits a test configuration $(\mathcal{S},\mathcal{B})$ with ${\rm Fut}(\mathcal{S},\mathcal{B})=0$.  We still take the cone construction of a K-semistable log Fano pair as before. The special test configuration determines a ray $v_{t}$ of valuations in $\Val_{X,x}$, emanating from the canonical component $v_0=\ord_{S_{\infty}}$.  Using the fact that the Futaki invariant is 0, a minimal model program argument shows that this implies for $t\ll 1$, $v_{t}$ is automatically $\mathbb{C}^*$-equivariant, which immediately implies the test configuration is  $\mathbb{C}^*$-equivariant. Therefore we show the following result (also see \cite{CS16} for an earlier attempt). 
\begin{thm}[\cite{LX16, LWX18}]
The K-semistability (resp. K-polystability) of a log Fano pair $(S,B)$ is equivalent to the $T$-equivariant K-semistablity ($T$-equivariant K-polystablity) for any torus group $T$ acting on $(S,B)$. 
\end{thm} 
For other groups $G$, e.g. finite groups or general reductive groups, we haven't proved the corresponding result as \eqref{e-Tkollar}. It is a consequence of the uniqueness part of the stable degeneration conjecture. We also note that in \cite{LX17}, it is proved that quasi-monomial minimizers over a $T$-equivariant klt singularity are automatically $T$-invariant.

\subsection{Donaldson-Sun's Conjecture}\label{ss-dsc}

One major application of what we know about the stable degeneration conjecture, formulated in Theorem \ref{t-high}, is the solution of \cite[Conjecture 3.22]{DS17} (see Conjecture \ref{conj-DS}), which predicts  that for a singularity appearing on a Gromov-Hausdorff limit of K\"ahler-Einstein metrics, its metric tangent cone only depends on the algebraic structure of the singularity.  In this section, we briefly explain the idea. 

\subsubsection{K-semistable degeneration}
Let $(M_k, g_k)$ be a sequence of K\"{a}hler-Einstein manifolds with positive curvature. Then possibly taking a subsequence, $(M_k, g_k)$ converges in the Gromov-Hausdorff topology to a limit
metric space $(X, d_\infty)$. By the work of Donaldson-Sun and Tian, $X$ is homeomorphic to a $\bQ$-Fano variety. For any point $x\in X$, a metric tangent cone $C_xX$ is defined as a pointed Gromov-Hausdorff limit:
\begin{equation}
C_xX=\lim_{r_k\rightarrow 0} \left(X, x, \frac{d_\infty}{r_k}\right).
\end{equation} 
By Cheeger-Colding's theory, $C_xX$ is always a metric cone. By \cite{CCT02}, the real codimension of singularity set of $C_xX$ is at least 4 and the regular part admits a Ricci-flat K\"{a}hler cone structure. In \cite{DS17}, it is further proved that $C_xX$ is an affine variety with an effective torus action. They proved that $C_xX$ is uniquely determined by the metric structure $d_\infty$ and can be obtained in the following steps. In the first step, they defined a filtration $\{\cF^\lambda\}_{\lambda\in \mathcal{S}}$ of the local ring $R=\cO_{X,o}$ using the limiting metric structure $d_\infty$. Here $\mathcal{S}$ is a set of positive numbers that they called the holomorphic spectrum which depends on the torus action on the metric tangent cone $C$. In the second step, they proved that the associated graded ring of $\{\mathcal{F}^\lambda\}$ is finitely generated and hence defines an affine variety, denoted by $W$. In the last step, they showed that $W$ equivariantly degenerates to $C$. Notice that this process depends crucially on the limiting metric $d_\infty$ on $X$. They then made the following conjecture. 
\begin{conj}[Donaldson-Sun]\label{conj-DS}
Both $W$ and $C$ depend only on the algebraic germ structure of $X$ near $x$.
\end{conj}
We made the following observations:
\begin{enumerate}
\item $\{\cF^{\lambda}\}$ comes from a valuation $v_0$. This is due to the fact that $W$ is a normal variety. More explicitly, since the question is local, we can assume $X={\rm Spec}(R)$ with the germ of $x\in X$, by the work in \cite{DS17}, one can embed both $X$ and $C$ into a common ambient space $\bC^N$, and $v_0$ on $X$ is induced by the monomial valuation $\wt_{\xi_0}$ where $\xi_0$ is the linear holomorphic vector field with $2{\rm Im}(\xi_0)$ being the Reeb vector field of the Ricci flat K\"{a}hler cone metric on $C$. By this construction, it is clear that the induced valuation by $v_0$ on $W$ is nothing but $\wt_{\xi_0}$.  

\item 
$v_0$ is a quasi-monomial valuation. This follows from Lemma \ref{lem-quasi}.
\end{enumerate}

More importantly we conjectured in \cite{Li15} that $v_0$ can be characterized as the unique minimizer of $\hvol_{X, x}$. As a corollary of the theory developed so far, we can already confirm \cite[Conjecture 3.22]{DS17} for $W$. 
\begin{thm}[\cite{LX17}]
The semistable cone $W$ in Donaldson-Sun's construction depends on the algebraic structure of $(X, x)$.
\end{thm}

The proof consists of the following steps consisting of analytic and algebraic arguments:
\begin{enumerate}
\item By Theorem \ref{thm-RF2K}, $(C, \xi_0)$ is K-polystable and in particular K-semistable. By Theorem \ref{t-SDChigh}, $\wt_{\xi_0}$ is a minimizer of $\hvol_{C}$.

\item By Proposition \ref{p-semideg}, $(W, \xi_0)$ is K-semistable. By Theorem  \ref{t-SDChigh} again, $\wt_{\xi_0}$ is a minimizer of $\hvol_W$. Moreover, by Theorem \ref{t-high}, $v_0$ is a minimizer of $\hvol_X$. 

\item $v_0$ is a quasi-monomial minimizer of $\hvol_X$ with a finitely generated associated graded ring. By Theorem \ref{t-high}, such a $v_0$ is indeed the unique minimizer of $\hvol$ among all quasi-monomial valuations.

\end{enumerate}

The following is an immediate consequence of Theorem  \ref{t-SDChigh}.
\begin{prop}\label{p-semideg}
Assume there is a special degeneration of a log-Fano cone $(X, D, \xi_0)$ to $(X_0, D_0, \xi_0)$. Assume that $(X_0, D_0,\xi_0)$ is K-semistable, then $(X, D, \xi_0)$ is also K-semistable, or equivalently, $\wt_{\xi_0}$ is the minimizer of $\hvol_{(X,D,x)}$.
\end{prop}

Asssume $(X, x)$ lives on a Gromov-Hausdorff limit of K\"{a}hler-Einstein Fano manifold. Then we can define the volume density in the sense of Geometric Measure Theory as the following quantity:
\begin{equation}
\Theta(x,X)=\lim_{r\rightarrow 0}\frac{{\rm Vol}(B_r(x))}{r^{2n} {\rm Vol}(B_1(\underline{0})}.
\end{equation}
Note that $n^n=\hvol(0, \bC^n)$. The normalized volumes of klt singularities on Gromov-Hausdorff limits have the following differential geometric meaning:
\begin{thm}[\cite{LX17}]\label{thm-hvol2Theta}
With the same notation as above, we have the identity:
\begin{equation}
\frac{\hvol(x,X)}{n^n}=\Theta(x,X).
\end{equation}
\end{thm}
\begin{proof}
From the standard metric geometry, we have $\Theta(x, X)=\Theta(o_C, C)$. Because $C$ admits a Ricci-flat K\"{a}hler cone metric,  by Theorem \ref{thm-RF2K}, $(C, \xi_0)$ is K-semistable. $\hvol(x, X)=\hvol(o_C, C)$.

On the other hand, since $C$ is a metric cone, from the definition of the volume of $\xi_0=\frac{1}{2}\left(r\partial_r-i J(r\partial_r)\right)$ is equal to:
\begin{eqnarray*}
\Theta(o_C, C)=\frac{{\rm Vol}(C\cap \{r=1\})}{{\rm Vol}(S^{2n-1})}=\vol(\xi_0).
\end{eqnarray*}
By Theorem \ref{thm-irSE}, $A(\wt_{\xi_0})=n$ and $\hvol(o_C, C)=n^n \vol(\xi_0)=n^n \Theta(o_C, C)$.

\end{proof}

\subsubsection{Uniqueness of polystable degeneration}

To confirm Donaldson-Sun's conjecture, we also need to prove the uniqueness of polystable degenerations for K-semistable Fano cones.

Since  a Fano cone singularity $(C,\xi)$ with a Ricci-flat K\"ahler cone metric is aways K-polystable (see \cite[Theorem 7.1]{CS15} and also Theorem \ref{thm-RF2K}), once knowing that $W$ only depends on the algebraic structure of $o\in M_{\infty}$, an affirmative answer to Conjecture \ref{conj-DS} follows from the following more general result by letting $(X,D, \xi_0)=(W, \emptyset, \xi_0)$: 

\begin{thm}[\cite{LWX18}]\label{t-uniquecone}
Given a K-semistable log Fano cone singularity $(X,D,\xi_0)$, there always exists a special test configuration $(\mathcal{X},\mathcal{D}, \xi_0; \eta)$ which degenerates $(X,D,\xi_0)$ to a K-polystable log Fano cone singularity $(X_0,D_0,\xi_0)$. Furthermore, such $(X_0,D_0,\xi_0)$ is uniquely determined by $(X, D, \xi_0)$ up to isomorphism.
\end{thm}

For the special case of smooth (or $\bQ$-Gorenstein smoothable) Fano varieties, this was proved in \cite[7.1]{LWX} based on analytic results which also show the uniqueness of Gromov-Hausdorff limit for a flat family of Fano K\"ahler-Einstein manifolds.
Our proof of Theorem \ref{t-uniquecone} is however a completely new algebraic argument.

We briefly discuss the idea to prove Theorem \ref{t-uniquecone} in \cite{LWX18}, which heavily depends on the study of normalized volumes as discussed in Section \ref{ss-cone}.

Let $(\cX^{(i)}, D^{(i)}, \xi_0, \eta^{(i)}), (i=1,2)$, be two special test configurations of the log Fano cone $(X, D, \xi_0)$ with the central fibre $(X^{(i)}_0, D^{(i)}_0, \xi_0)$. To show Theorem \ref{t-uniquecone}, the main step is to show that if $\Fut(\cX^{(i)}, \cD^{(i)}, \xi_0; \eta^{(i)})=0, (i=1,2)$, then there exists special test configurations $(\cX'^{(i)}, \cD'^{(i)})$ of $(X^{(i)}_0, D^{(i)}_0)$ such that $(\cX'^{(i)}, \cD'^{(i)})$ have isomorphic central fibres, which we will describe below. 

We consider the normalized volume functional defined on the valuation space $\Val_{X,x}$ over the vertex $x$ of the cone $X$. Then $(\cX^{(1)}, \cD^{(1)}, \xi_0; \eta^{(1)})$ determines a  ``ray" of valuations emanating from the toric valuation  $\wt_{\xi_0}$ and the generalized Futaki invariant ${\rm Fut}(\cX^{(1)},\cD^{(1)}, \xi_0; \eta^{(1)})$ is the derivative of the normalized volume at $\wt_{\xi_0}$ along this ray.

\begin{equation}
\xymatrix@1 @R=1.2pc @C=1pc
{
(X^{(2)}_0, D^{(2)}_0) \ar@{~>}_{(\cX'^{(2)}, \cD'^{(2)})}[dd]
 &   & &  &(X, D) \ar_{ (\cX^{(2)}, \cD^{(2)})\longleftarrow \cY^{(2)}_k\longleftarrow\cE^{(2)}_k}@{~>}[llll] \ar^{(\cX^{(1)}, \cD^{(1)})\leftarrow \cY_k\leftarrow \cE_k=E_k\times\bC^1}@{~>}[dd]  & \ar[l] Y_k\leftarrow E_k
 \\
& & &   \hskip 2cm &   &\\
(X'_{0}, D'_{0})  & & & & (X^{(1)}_0, D^{(1)}_0)\ar@{~>}_{(\cX'^{(1)}, \cD'^{(1)})}[llll] & \ar[l] Y_{k,0} \leftarrow E_k
}
\end{equation}

We can approximate $\xi_0$ by a sequence of integral vectors $\tilde{\xi}_{k}$ such that $|\tilde{\xi}_{k}-k \xi_0|\le C$. 
For $k\gg 1$, the vector $\tilde{\xi}_k-\eta$ corresponds to a Koll\'{a}r component $E_k$ over $X$.
Our key argument is to show that $E_k$ can be degenerated along $(\cX^{(2)}, \cD^{(2)})$ to get a model $\cY_k^{(2)}\rightarrow \cX^{(2)}$ with an exceptional divisor $\cE^{(2)}_k$ such that $(\cY_k^{(2)}, \cE^{(2)}_k)\times_{\bC}\bC^*\cong (Y_k, E_k)\times \bC^*$ where the isomorphism is compatible with the equivariant isomorphism of the {\it second} special test configuration. Note that $E_k\times\bC^*$ determines a divisorial valuation over $X\times\bC^*$ and hence over $(\cX^{(2)}, \cD^{(2)})$. So the goal is to show that this divisorial valuation can be extracted as the only exceptional divisor over $\cX^{(2)}$. By the work in the minimal model program (MMP) (see \cite{BCHM10}), this would be true if there is 
a graded sequence of ideals $\frak{A}_\bullet$ and a positive real number $c'_k$ such that two conditions are satisfied: 
$$(\cX^{(2)}, \cD^{(2)}+c'_k \frak{A}_\bullet)\mbox{ is klt \ \ \ and\ \ \ } A(E_k\times\bC; \cX^{(2)}, \cD^{(2)}+ c'_k \frak{A}_\bullet)<1,$$
where $A(E_k\times\bC; \cX^{(2)}, \cD^{(2)}+c'_k \frak{A}_\bullet)$ is the log discrepancy of (the birational transform of) $E_k\times\bC$ with respect to the triple $(\cX^{(2)}, \cD^{(2)}+ c'_k \frak{A}_\bullet)$. 

To find such a $\frak{A}_\bullet$, we look at the graded sequence of valuative ideals $\{\fa_{\bullet}\}$ of $\ord_{E_k}$ and its equivariant degeneration along the second special test configuration $(\cX^{(2)}, \cD^{(2)})$. The resulting graded sequence of ideals over $\cX^{(2)}$ will be denoted by $\frak{A}_{\bullet}$. 
Using the study in Section \ref{ss-cone} one can show the assumptions that  $(\cX^{(1)}, \cD^{(1)}; \xi_0)$ is K-semistable and $\Fut(\cX^{(1)}, \cD^{(1)}, \xi_0; \eta)=0$ implies 
\[
\mbox{$f(k):=\hvol(E_k)$ is of the order $f(0)+O(k^{-2})$. }
\]
This in turn guarantees that we can find $c'_k$ satisfying the above two conditions. 

Applying the relative Rees algebra construction to $\cE^{(2)}_k\subset \mathcal{Y}^{(2)}_k/\mathbb{C}$,  we get a family over $\mathbb{C}^2$, which over $\mathbb{C}\times \{t\}$ is the same as $(\cX^{(1)},\cD^{(1)})$ for $t\neq 0$ and gives a degeneration of $(X^{(1)}_0, D^{(1)}_0)$ for $t=0$. On the other hand, over $\{0\}\times \mathbb{C}$, we get a degeneration of $(X^{(2)}_0, D^{(2)}_0)$. Therefore, we indeed show that the two special fibers of two special test configurations $(\cX^{(i)},\cD^{(i)}, \xi_0; \eta^{(i)})\ (i=1,2)$ with ${\rm Fut}(\cX^{(i)},\cD^{(i)}, \xi_0; \eta^{(i)})=0$  will have a common degeneration.

\subsection{Estimates in dimension three and K-stability of threefolds}\label{ss-cubic3}

In general, it is not so easy to find the minimizer of $\hvol(\cdot)$ for a given singularity. A number of cases have been computed in \cite{Li15, LL16, LX16, LX17, LiuX17} including quotient singularities, ADE singularities in all dimensions (except 4-dimensional $D_4$) etc. 

Here we study normalized volumes of threefold klt singularities, and then give a global application where we show that all GIT semi-stable (resp. polystable) cubic threefolds are also 
K-semi-stable (resp. K-polystable). Our main estimate is in Theorem \ref{thm:3dimhvol}, which heavily depends on classifications of canonical threefold singularities. 

\begin{thm}[\cite{LiuX17}]\label{thm:3dimhvol}
 Let $x\in X$ be a $3$-dimensional non-smooth klt singularity.
 Then $\hvol(x, X)\leq 16$ and the equality holds if and only
  if it is an $A_1$ singularity;
\end{thm}

The proof of Theorem \ref{thm:3dimhvol} heavily relies on the classification theory of three dimensional canonical and terminal singularities, developed in the investigation of explicit three dimensional MMP. 

The idea  goes as follows.
Firstly, we reduce to the case of Gorenstein canonical singularity.
If $x\in X$ is not Gorenstein, let us take the index one cover $\tilde{x}\in\widetilde{X}$ 
of $x\in X$. Hence $\tilde{x}\in\widetilde{X}$ is a Gorenstein
canonical singularity. If $\tilde{x}\in\widetilde{X}$
is smooth, then $\hvol(x, X)=27/\ind(x,K_X)\leq 13.5<16$.
If $\tilde{x}\in\widetilde{X}$ is not smooth, the
a weak version of finite degree formula (Proposition
\ref{prop:weakfinitedeg}) implies
that $\hvol(x, X)<\hvol(\tilde{x},\widetilde{X})$.

Next, let us assume that $x\in X$ is Gorenstein 
canonical. By \cite[Proposition 2.36]{KM98}, there
exist only finitely many crepant exceptional divisors over
$X$. By \cite{BCHM10}, we can extract these divisors
simultaneously on a birational model $Y_1\to X$. If none of these exceptional
divisors are centered at $x$, then \cite[Theorem 5.34]{KM98}
implies that $x\in X$ is a cDV singularity, hence $\lct(\fm_x)
\leq 4-\mult(\fm_x)$ which implies 
\[
 \hvol(x, X)\leq \lct(\fm_x)^3\mult(\fm_x)\leq (4-\mult(\fm_x))^3
 \mult(\fm_x)\leq 16.
\]
The equality case can be characterized using the volume of birational
models approach in \cite{LX16}.
If some crepant exceptional divisor $E_1\subset Y_1$ is 
centered at $x$, then let us run $(Y_1,\epsilon E_1)$-MMP
over $X$ for $0\ll \epsilon<1$. By \cite[1.35]{Kol13},
this MMP will terminate as $Y_1\dashrightarrow Y\xrightarrow[]{g}Y'$,
where $Y_1\dashrightarrow Y$ is the composition of a sequence
of flips, and $g : Y \to Y'$ contracts the birational transform $E$ of $E_1$.
If $g(E)$ is a curve, then $Y'$ has cDV singularities
along $g(E)$ by \cite[Theorem 5.34]{KM98}. By choosing 
a point $y'\in g(E)$, we have 
$$\hvol(x, X)<\hvol(y',Y')\leq 16.$$
If $g(E)=y'$ is a point, then we still have $\hvol(x, X)<\hvol(y',Y')$.
Thus it suffices to show $\hvol(y',Y')<16$. 

If $Y$ has
a singular point $y\in E$, then we know that $y\in Y$ is 
a cDV singularity. Hence 
$$\hvol(y',Y')<\hvol(y,Y)\leq 16.$$
So we may assume that $Y$ is smooth along $E$. In particular,
$E$ is a (possibly non-normal) reduced Gorenstein del Pezzo
surface. If $E$ is normal, then classification of such 
surfaces show that $(-K_E)^2\leq 9$. Thus 
\[
\hvol(y',Y')\leq A_{Y'}(\ord_E)^3 \vol(\ord_E)=(-K_E)^2\leq 9<16.
\]
If $E$ is non-normal, then from Reid's classification \cite{Rei94}
either $(-K_E)^2\leq 4$ or the normalization of $E$ is a 
Hirzebruch surface. In the former case, we have $\hvol(y',Y')\leq 4$.
In the latter case, we need to take a general fiber $l$ of $E$ and
argue that $\hvol_{Y',y'}(\ord_l)\leq 16$.

Here are some intermediate results in proving Theorem \ref{thm:3dimhvol}.

\begin{prop}
 Let $\phi : (Y, y) \to (X, x)$ be a birational morphism of klt singularities such that $y \in \mathrm{Ex}(\phi)$. If $K_Y\leq\phi^*K_X$, then $\hvol(x,X) < \hvol(y, Y)$.
\end{prop}

See Conjecture \ref{c-fdf} for more discussions about the following Proposition.

\begin{prop}\label{prop:weakfinitedeg}
 Let $\pi:(\widetilde{X},\tilde{x})\to (X,x)$ be a finite quasi-\'etale
  morphism of klt singularities of degree at least $2$. Then we have
 \[
  \hvol(x,X)<\hvol(\tilde{x},\widetilde{X})\leq \deg(\pi)\cdot\hvol(x,X).
 \]
\end{prop}

\bigskip

As mentioned \cite{SS17}, one main application of the local volume estimate
Theorem \ref{thm:3dimhvol} is to the K-stability question of cubic threefolds. 
\begin{thm}[\cite{LiuX17}]\label{thm:Kcubic}
 A cubic threefold is K-(poly/semi)stable if and only
 if it is GIT (poly/semi)stable. In particular,
 any smooth cubic threefold is K-stable.
\end{thm}

The general strategy to prove Theorem \ref{thm:Kcubic} is via the comparison of moduli spaces which has first appeared in \cite{MM93} built on the work of \cite{Tia92}. Later it was also applied in \cite{OSS, SS17}.

First, one can construct a proper algebraic space 
which is a good quotient moduli
space with closed points parametrizing all
smoothable K-polystable $\bQ$-Fano varieties (see e.g. \cite{LWX, Oda15}).  Let $M$ be the closed subspace whose closed points
parametrize KE cubic threefolds and their K-polystable
limits. By \cite{Tia87}, we know that at least one cubic threefold,
namely the Fermat cubic threefold, admits a KE metric. 
Hence $M$ is non-empty. By the Zariski openness of K-(semi)stability
of smoothable Fano varieties (cf. \cite{Oda15,  LWX}), the K-moduli space
$M$ is birational to the GIT moduli space $M^{\rm GIT}$ of cubic threefolds.

Next, we will show that any K-semistable limit $X$ of
a family of cubic threefolds $\{X_t\}$ over a punctured curve is necessarily a cubic threefold.
The idea is to control the singularity of $X$ use an inequality from \cite{Liu16} (see
Theorem \ref{thm:liuvolcomp}) between the global volume of a K-semistable
Fano variety and the local normalized volume.
Since the volume of $X$ is the same as the volume of 
a cubic $3$-fold which is $24$, Theorem \ref{thm:liuvolcomp}
immediately implies that $\hvol(x,X)\geq \frac{81}{8}$ for
any closed point $x\in X$. The limit $X$ carries
a $\bQ$-Cartier Weil divisor $L$ which is the flat limit
of hyperplane sections in the cubic threefolds $X_t$.
It is clear that $-K_X\sim_{\bQ}2L$ and $(L^3)=3$, thus
once we show that $L$ is Cartier, we can claim that
$X$ is a cubic threefold using a result of T. Fujita.

Assume to the contrary that $L$ is not Cartier
at some point $x\in X$, then we may take the index $1$ cover
$(\tilde{x}\in \widetilde{X})\to (x\in X)$ of $L$. From
the finite degree formula Theorem \ref{t-fdf}, 
\[
 \hvol(\tilde{x},\widetilde{X})=\ind(L)\cdot\hvol(x,X)\geq 81/4.
\]
Hence $\tilde{x}\in\widetilde{X}$ is a smooth point and 
$\ind(L)=2$ by Theorem \ref{thm:3dimhvol}. Thus $x\in X$ is a
quotient singularity of type $\frac{1}{2}(1,1,0)$ from
the smoothable condition. Then using the local Grothendieck-Lefschetz
theorem, we can show that $L$ is indeed Cartier at $x\in X$ which
is a contradiction.

So far we have shown that any K-polystable point $X$ in $M$
is a cubic threefold. By an argument of Paul and Tian in \cite{Tia94},
we know that any K-(poly/semi)stable hypersurface is GIT (poly/semi)stable.
Thus we obtain an injective birational morphism
$M\to M^{\rm GIT}$ between proper algebraic spaces.
This implies that $M$ is isomorphic to $M^{\rm GIT}$
which finishes the proof.

\begin{thm}[\cite{Liu16}]\label{thm:liuvolcomp}
 Let $X$ be an $n$-dimensional K-semistable Fano variety. Then for any
 closed point $x\in X$, we have
 \[
  (-K_X)^n\leq \left(1+\frac{1}{n}\right)^n\hvol(x,X).
 \]
\end{thm}

When $X$ is smooth, the above result was first proved in \cite{Fuj18}.

 
\section{Questions and Future research}\label{s-ques}
\subsection{Revisit stable degeneration conjecture}
The following two parts of stable degeneration conjecture, proposed in \cite{Li15}, are still missing.
\begin{conj}[Quasi-monomial] Let $x\in (X, D)$ be a klt singularity. Any minimizer of $\hvol_{(X,D),x}$ is quasi-monomial. 
\end{conj}

\begin{conj}[Finite generation] Let $x\in (X={\rm Spec}(R),D)$ be a klt singularity. Any minimizer of $\hvol_{(X,D),x}$ has its associated graded ring $\gr_v(R)$ to be finitely generated. 
\end{conj}
Due to the fundamental role of the stable degeneration conjecture, it implies many other interesting properties. We discuss a number of special cases or consequences, with the hope that some of them might be solved first. 

One interesting consequence of the uniqueness of the minimizer is the following
\begin{conj}[Group action]\label{g-equiva}
If  there is a group $G$ acting on the klt singularity $x\in (X, D)$ such that $x$ is a fixed point, then 
there exists a $G$-invariant minimizer. 
\end{conj}
Applying  this conjecture to a cone singularity, it implies that to test the K-semistability of a log Fano $(S,B)$ with a $G$-action, we only need to test on $G$-equivariant test configurations, a fact known for a Fano manifold $X$ and $G$-reductive. 

There are two special cases naturally appearing in contexts. The first one is that when $G$ is a torus group $T$.  It follows the argument in \cite{Blu18} and the techniques of degenerating ideals to their initials, that there is a $T$-equivariant minimizer. This is the philosophy behind Section \ref{ss-equiva}.   It also follows from \cite{LX17} that any quasi-monomial minimizer is $T$-equivariant.

A more challenging case is when $G$ is a finite group. Indeed, Conjecture \ref{g-equiva} for finite group $G$ implies the following finite degree formula.

\begin{conj}[Finite degree formula]\label{c-fdf}
 If $\pi\colon (y\in Y,D')\to (x\in X,D)$ is a dominant finite morphism between klt singularities, such that $K_Y+D'=\pi^*(K_X+D)$, then
$$\deg(\pi)\cdot \hvol(x, X,D)=\hvol(y, Y,D').$$
\end{conj}
This is useful when we want to bound the klt singularities $x\in (X,D)$ with a large volume. 
\begin{thm}[{\cite{LX17}}]\label{t-fdf}  
Conjecture \ref{c-fdf} is true  when $(X, x)$ is on a Gromov-Hausdorff limit of K\"{a}hler-Einstein Fano manifolds. 
\end{thm}
\begin{proof}
Let $\pi=\pi_X: (Y,y)\rightarrow (X,x)$ be a quasi-\'{e}tale morphism, i.e. $\pi_X$ is \'{e}tale in codimension one. 
Then $\pi_X$ induces a quasi-\'{e}tale morphism along the 2-step degeneration of $X$.

\begin{equation}
\xymatrix@1 @R=1.2pc @C=1pc
{
Y \ar@{~>}_{}[r]
 \ar^{\pi_X} [d] & W_Y \ar@{~>}_{}[r] \ar^{\pi_W}[d] &  C_Y  \ar^{\pi_C}[d] \\
X \ar@{~>}_{}[r] & W \ar@{~>}_{}[r] & C
}
\end{equation}
We can use the above diagram to prove the degree multiplication formula. Roughly speaking, because $C$ admits a Ricci-flat K\"{a}hler cone metric $\omega_C$ with radius function $r^2$ and $\pi_C$ is quasi-\'{e}tale, we can pull back it to get $\pi_C^*r^2$ which is also a potential for a weak Ricci-flat K\"{a}hler cone metric $\omega_{C_Y}$. By Theorem \ref{thm-RF2K}, Theorem \ref{t-SDChigh} and Theorem \ref{t-high}, we know that the Reeb vecotor field associated to $\omega_C$ (resp. $\omega_{C_Y}$) induces minimizing valuations of $\hvol_{X}$ (resp. $\hvol_Y$).
So we get 
\begin{eqnarray*}
\hvol(y, Y)&=&\hvol(o_{C_Y}, C_Y)=\deg(\pi_C)\cdot \hvol(o_C, C)=\deg(\pi_C) \cdot \hvol(x, X).
\end{eqnarray*}
\end{proof}

Another consequence of the stable degeneration conjecture is the following strengthening of Theorem \ref{t-lower}. 

\begin{conj}\label{c-constru}
Let $\pi:(\cX,\cD)\to T$ together with a section $t\in T\mapsto x_t\in\cX_t$
 be a $\bQ$-Gorenstein flat family of klt singularities. Then
 the function $t\mapsto \hvol(x_t,\cX_t,\cD_t)$  is construtible with respect to the Zariski topology.
\end{conj}
Besides the stable degeneration conjecture, to prove Conjecture \ref{c-constru}, we also need to know the well expected speculation that K-semistability is an open condition. It is also natural to consider the volume of non-closed point. However, the following conjecture says after the right scaling, it does not contribute more information. 

\begin{conj}If a klt pair $(X,D)$ has a non-closed point $\eta$, and let $Z=\overline{\{\eta\}}$ has dimension $d$. Pick a general closed point $x\in Z$, then
$$\hvol(x, X,D)=\hvol(\eta,X,D)\cdot \frac{n^n}{(n-d)^{n-d}}.$$
\end{conj}
In fact, combining the argument in \cite{LZ18},  for any valuation $v\in \Val_X$ such that its center $Z={\rm Center}_X(v)$ on $X$ is of dimension $d$ and $x\in Z $, denoted by $\eta$ is the generic point of $Z$, one can show that
$$\frac{\hvol_{(X,D),\eta}(v)\cdot n^n}{(n-d)^{n-d}}\ge \hvol(x, X,D).$$ i.e.,
$$\hvol(x, X,D)=\inf_{v}\left\{ \frac{n^n\cdot \hvol_{(X,D),\eta}(v)}{(n-d)^{n-d}}\ | \ x\in Z=\overline{\{\eta\}}={\rm Center}_X(v), \dim(Z)=d \right\}.$$


\subsection{Birational geometry study}
A different invariant attached to a klt singularities, called the minimal log discrepancy has been intensively studied in the minimal model program, though there are still many deep questions unanswered.  We can formulate many similar questions for $\hvol$.

\subsubsection{Inversion of adjunction}
One could look for a theory of the change of the volumes when the klt pair is `close' to a log canonical singularities, using the inversion of adjunction.  We have some results along this line. 

\begin{prop}
 Let $x\in (X,\Delta)$ be an $n$-dimensional klt singularity.
 Let $D$ be a normal $\bQ$-Cartier divisor containing
 $x$ such that $(X,D+\Delta)$ is plt. Denote by $\Delta_D$ the different
 of $\Delta$ on $D$. Then 
 \[
  \lim_{\epsilon\to 0+}\frac{\hvol(x, X,(1-\epsilon)D+\Delta)}{n^n\epsilon}
  =\frac{\hvol(x, D,\Delta_D)}{(n-1)^{n-1}}.
 \]
 
\end{prop}

\begin{proof}
 Using the degeneration argument in \cite{LZ18}, 
 we know that 
 \[
 \epsilon^{-1} \hvol(x, X,(1-\epsilon)D+\Delta)\geq \frac{n^n}{(n-1)^{n-1}}\hvol(x, D,\Delta_D).
 \]
 Hence it suffices to show the reverse inequality
 is true after taking limits. Let us pick an arbitrary
 Koll\'ar component $S$ over $x\in (D,\Delta_D)$ with valuation ideals
 $\fa_m:=\fa_m(\ord_S)$. Choose $m$ sufficiently 
 divisible so that $\fa_{im}=\fa_m^i$ for any $i\in\bN$.
 Then we know that $\lct(D,\Delta_D;\fa_m)=A_X(\ord_S)/m=:c$.
 Let $\fb_m$ be the pull-back ideal of $\fa_m$ on $X$.
 By inversion of adjunction, we have $\lct(X,D+\Delta;\fb_m)=\lct(D,\Delta_D;\fa_m)=c$.

 Let $E$ be an exceptional divisor over $X$ computing
 $\lct(X,D+\Delta;\fb_m)$. Then $E$ is centered at $x\in X$
 since $(X,D+\Delta)$ is plt. For $\epsilon_1>0$ sufficiently small,
 we have that $(X,\Delta+(1-\epsilon_1)(D+c\cdot\fb_m))$
 is a klt pair over which the discrepancy of $E$ is negative.
 Thus \cite{BCHM10} implies that there exists a proper birational model
 $\mu:Y\to X$ which only extracts $E$. Moreover,
 $\mu:Y\to X$ is a log canonical modification of 
 $(X,\Delta+D+c\cdot\fb_m)$.
 Let $\widetilde{D}$ be the normalization of $\mu_*^{-1}D$.
 Then by adjunction, the lifting morphism
 $\tilde{\mu}:\widetilde{D}\to D$ is a log
 canonical (in fact plt) modification of $(D,\Delta_D+c\cdot\fa_m)$.
 Since $\mathrm{Bl}_{\fa_m}D\to D$ provides a model of the Koll\'ar component
 $S$, this is the only log canonical modification of 
 $(D,\Delta_D+c\cdot\fa_m)$. Hence $E|_{\widetilde{D}}=S$ and 
 $(\widetilde{D},\tilde{\mu}_*^{-1}\Delta_D+
 E|_{\widetilde{D}})$ is plt. Then by inversion
 of adjunction, $(Y,\mu_*^{-1}\Delta+\mu_*^{-1}D+E)$
 is qdlt and $\mu_*^{-1}D=\widetilde{D}$ is normal. 
 Note that all the constructions so far are independent of the
 choice of $\epsilon$.
 
 Over the qdlt model $(Y,\mu_*^{-1}\Delta+\mu_*^{-1}D+E)$, we consider a quasi-monomial valuation $v_{\lambda}$
 of weights $1$ and $\lambda$ along divisors $\widetilde{D}$
 and $E$ respectively. By adjunction, we know that $A_{(X,\Delta)}(\ord_E)=
 A_{(D,\Delta_D)}(\ord_S)+\ord_E(D)$. Hence computation shows that
 \[
  A_{(X,\Delta+(1-\epsilon) D)}(v_{\lambda})
  =\lambda A_{(D,\Delta_D)}(\ord_S)+\lambda\epsilon \cdot \ord_E(D)+\epsilon.
 \]
 Then using the  Okounkov body description of the volume (see \cite{LM09, KK12}), we easily see 
 that $\vol(v_\lambda)\leq \lambda^{1-n}\vol(\ord_S)$.
 Hence
\begin{eqnarray*}
  \hvol_{(X,\Delta+(1-\epsilon)D)}(v_\lambda)&  \leq &\lambda^{1-n}((A_{(D,\Delta_D)}(\ord_S)+\epsilon\cdot \ord_E(D))\lambda+\epsilon)^n\vol(\ord_S)\\
  &=:&\phi(\lambda).
 \end{eqnarray*}
 It is easy to see that $\phi(\lambda)$ reaches
 its minimum at 
 $$\lambda_0=\frac{(n-1)\epsilon}{A_{D,\Delta_D}(\ord_S)+\epsilon\cdot \ord_E(D)}.$$
 Hence computation shows
 \[
  \epsilon^{-1}\hvol_{(X,\Delta+(1-\epsilon)D)}(v_{\lambda_0})
  \leq \frac{n^n}{(n-1)^{n-1}}(A_{(D,\Delta_D)}(\ord_S)+\epsilon\cdot \ord_E(D))^{n-1}\vol(\ord_S).
 \]
 Thus
 \[
  \limsup_{\epsilon\to 0}\epsilon^{-1}\hvol(x, X,\Delta+(1-\epsilon)D)
  \leq \frac{n^n}{(n-1)^{n-1}}\hvol_{(D,\Delta_D)}(\ord_S)
 \]
 Since this inequality holds for any Koll\'ar component
 $S$ over $x\in (D,\Delta_D)$, the proof is finished.
\end{proof}

When the center is zero dimensional, we also have

\begin{prop}
 Let $x\in (X,\Delta)$ be a klt singularity. Let $D\geq 0$
 be a $\bQ$-Cartier divisor such that $(X,\Delta+D)$ is log
 canonical with  $\{x\}$ being the minimal non-klt center.
 Then there exists $\epsilon_0>0$ (depending only on
 the coefficient of $\Delta,D$ and $n$) and a quasi-monomial
 valuation $v\in\Val_{X,x}$ such that $v$ computes
 both $\lct(X,\Delta;D)$ and  $\hvol(x,X,\Delta+(1-\epsilon)D)$ for any $0<\epsilon<\epsilon_0$.
 In particular, 
 \[
  \hvol(x, X,\Delta+(1-\epsilon)D)
 =\hvol_{x,(X,\Delta)}(v)\cdot\epsilon^n\textrm{ for any }
 0<\epsilon<\epsilon_0.
 \]
\end{prop}

\begin{proof}
Let $Y^{\dlt}\to X$ be a dlt modification
  of $(X,\Delta+D)$. Let $K_{Y^{\dlt}}+\Delta^{\dlt}$
  be the log pull back of $K_X+\Delta+D$. Then by \cite{dFKX},
  the dual complex $\DR(\Delta^{\dlt})$ form
  a natural subspace of $\Val_{X,x}^{=1}$. Any divisorial
  valuation $\ord_E$ computing $\lct(X,\Delta;D)$
  corresponds to a rescaling of a valuation in $\DR(\Delta^{\dlt})$.
  Consider the function $\vol_X:\DR(\Delta^{\dlt})\to\bR_{>0}
  \cup\{+\infty\}$. Denote by $\DR^{\circ}(\Delta^{\dlt})$
  the open subset of $\DR(\Delta^{\dlt})$ consisting
  of valuations centered at $x$. Since $\{x\}$ is the minimal
  non-klt center of $(X,\Delta+D)$, we know that $\DR^{\circ}(\Delta^{\dlt})$
  is non-empty.
  By \cite{BFJ14} the function $\vol$ is continuous
  on $\DR(\Delta^{\dlt})$, so we
  can take a $\vol$-minimizing valuation $v\in\DR^{\circ}(\Delta^{\dlt})$.
  Hence  $v$ is a minimizer of $\hvol$ restricted to
  $\DR^{\circ}(\Delta^{\dlt})$.
  
  Assume $S$ is an arbitrary Koll\'ar component over $(X,\Delta+(1-\epsilon)D)$.
  Then we have a birational morphism $\mu: Y\to X$ such that
  $K_Y+\mu_*^{-1}(\Delta+(1-\epsilon)D)+S$ is plt,
  and $\mu$ is an isomorphism away from $x$ with $S=\mu^{-1}(x)$.
  Then by ACC of lct \cite{HMX14}, we know that
  there exists $\epsilon_0$ such that $K_Y+\mu_*^{-1}(\Delta+D)+S$
  is log canonical whenever $0<\epsilon<\epsilon_0$.
  Let $v'$ be an arbitrary divisorial valuation in
  $\DR^{\circ}(\Delta^{\dlt})$. Since $K_Y+\mu_*^{-1}(\Delta+D)+S\sim_{\bQ}
  \mu^*(K_X+\Delta+D)+A_{(X,\Delta+D)}(\ord_S)S$, we have
  \[
   0\leq A_{(Y,\mu_*^{-1}(\Delta+D)+S)}(v')=A_{(X,\Delta+D)}(v')
   -A_{(X,\Delta+D)}(\ord_S)\cdot v'(S).
  \]
  Since $A_{(X,\Delta+D)}(v')=0$ and  $v'(S)>0$ since
  $\{x\}$ is the only lc center, we know that $A_{(X,\Delta+D)}(\ord_S)=0$.
  Thus a rescaling of $\ord_S$ belongs 
  to $\DR(\Delta^{\dlt})$. Then by \cite{LX16} we see
  that 
  \begin{align*}
  \hvol(x,X,\Delta+(1-\epsilon)D)&=\min_{v'\in\DR(\Delta^{\dlt})}\hvol_{(X,\Delta+(1-\epsilon)D)}(v')\\
  &=\epsilon^n\min_{v'\in\DR(\Delta^{\dlt})} \vol_{X}(v')=
  \hvol_{x,(X,\Delta)}(v)\cdot\epsilon^n.
  \end{align*}
\end{proof}

One should be able to solve the following question using the above techniques.

\begin{que}
Let $x\in (X,\Delta)$ be an $n$-dimensional klt singularity. Let $D$ be an effective $\bQ$-Cartier $\bQ$-Weil
divisor through $x$. Let $c=\lct(X,\Delta;D)$,
and let $W$ be the minimal log canonical
center of $(X,\Delta+cD)$ containing $x$. 
By Kawamata's subadjunction,
we have $(K_X+\Delta+cD)|_W=K_W+\Delta_W+J_W$, where $(W,\Delta_W+J_W)$ is a generalized klt
pair. Denote by $k:=\codim_{X}W$, then is it true that
\[
 \lim_{\epsilon\to 0+} \epsilon^{-k}\frac{\hvol(x,X,\Delta+(1-\epsilon)cD)}{n^n}\geq
 \frac{\hvol(w,X,\Delta)}{k^k}\cdot \frac{\hvol(x,W,\Delta_W+J_W)}{(n-k)^{n-k}}
\]
where $w$ is the generic point of $W$ in $X$ and $\hvol(x,W,\Delta_W+J_W)$ is similarly defined as for the usual klt pair case in Definition \ref{d-normvol}?
\end{que}

\subsubsection{Uniform bound}
The following is conjectured in \cite{SS17} (see also \cite{LiuX17}).
\begin{conj}Let $x\in X$ be an $n$-dimensional singular point, then $\hvol(x, X)\le 2(n-1)^n$.
\end{conj}
The constant $2(n-1)^n$ is the volume of a rational double point.
When $n=3$, it is proved in Theorem \ref{thm:3dimhvol}. The implication to the K-stability question of cubic hypersurfaces as in the argument of Theorem \ref{thm:Kcubic} holds in any dimension. 

We also ask whether the following strong property of the set of local volumes holds.  
\begin{ques}Fix the dimension $n$, and a finite set $I\subset [0,1]$. Is it true that the set ${\rm Vol}^{\rm loc}_{n,I}$ consisting of all possible local volumes of $n$-dimensional klt singularities $x\in (X,D)$ with $(\mbox{coefficients of }D)\subset I$ has the only accumulation point 0? 
\end{ques}

Next we give a comparison between local volumes and minimal
log discrepancies.

\begin{thm}\label{thm:mldhvol}
 Let $x\in (X,\Delta)$ be an $n$-dimensional complex klt singularity.
 Then there exists a neighborhood $U$ of $x\in X$
 such that $(U,\Delta|_U)$ is $(\hvol(x,X,\Delta)/n^n)$-lc.
 Moreover, $\mld(x,X,\Delta)>\hvol(x,X,\Delta)/n^n$.
\end{thm}

\begin{proof}
If $x\in X$ is not $\bQ$-factorial then
we may replace $X$ by its $\bQ$-factorial modification
under which the local volume will increase by
\cite[Corollary 2.11]{LiuX17}.
Let $\Delta_i$ be any component
of $\Delta$ containing $x$. Then \cite[Theorem 33]{BL18}
implies that $A_{(X,\Delta)}(\Delta_i)\geq \hvol(x,X,\Delta)/n^n$.
Let $E$ be any exceptional divisor over $X$ such that $x$ is contained
in the Zariski closure of $c_X(E)$ and $a(E;X,\Delta)<0$. 
Then by \cite[Corollary 1.39]{Kol13}, there exists
a proper birational morphism $\mu:Y\to X$ such that $Y$
is normal, $\bQ$-factorial and $E=\mathrm{Ex}(\mu)\supset\mu^{-1}(x)$. 
Since $K_Y+\mu_*^{-1}\Delta-a(E;X,\Delta)E=\mu^*(K_X+\Delta)$, we know that $(Y,\mu_*^{-1}\Delta-a(E;X,\Delta)E)$
is klt. Let $y\in \mu^{-1}(x)$ be a point, then $y$ lies on $E$.
Hence by \cite[Corollary 2.11]{LiuX17} and \cite[Theorem 33]{BL18}
we have 
\[
 \hvol(x,X,\Delta)< \hvol(y,Y,\mu_*^{-1}\Delta-a(E;X,\Delta)E)\leq A_{(X,\Delta)}(E)n^n.
\]
Thus $A_{(X,\Delta)}(E)>\hvol(x,X,\Delta)/n^n$ which finishes the proof.
\end{proof}

Next we will discuss application to boundedness 
generalizing a result by C. Jiang \cite[Theorem 1.6]{Jia17}.

\begin{cor}\label{cor:bdd}
 Let $n$ be a natural number and $c$ a positive
 real number. Then the projective varieties $X$ satisfying the
 following properties:
 \begin{itemize}
  \item $(X,\Delta)$ is a klt pair of dimension $n$
  for some effective $\bQ$-divisor $\Delta$,
  \item $-(K_X+\Delta)$ is nef and big,
  \item $\alpha(X,\Delta)^n(-(K_X+\Delta))^n\geq c$,
 \end{itemize}
 form a bounded family.
\end{cor}

\begin{proof}
 By \cite[Theorem A and D]{BJ17} (generalizing \cite{Liu16}), for any closed point
 $x\in X$ we have
 \[
  c\leq \alpha(X,\Delta)^n(-(K_X+\Delta))^n\leq 
  \delta(X,\Delta)^n(-(K_X+\Delta))^n
  \leq \left(1+\frac{1}{n}\right)^n\hvol(x,X,\Delta).
 \]
 Hence Theorem \ref{thm:mldhvol} implies $(X,\Delta)$
 is $(c/(n+1)^n)$-lc. Therefore, the BAB Conjecture
 proved by Birkar in \cite[Theorem 1.1]{Bir16}
 implies the boundedness of $X$.
\end{proof}

\begin{rem}
 In the conditions of Corollary \ref{cor:bdd} if we 
 also assume that the coefficients
 of $\Delta$ are at least $\epsilon$ for any fixed
 $\epsilon\in (0,1)$, then such pairs $(X,\Delta)$
 are log bounded. This partially generalizes \cite[Theorem 1.4]{Che18}.
  Besides, all results should hold for $\bR$-pairs.
\end{rem}

\begin{que}
 Is it true that for any $n$-dimensional klt singularity
 $x\in X$, we have $\mld(x,X)\geq \hvol(x,X)/n^{n-1}$?
\end{que}

\subsection{Miscellaneous  Questions}

\subsubsection{Positive characteristics}

In this section, we consider a variety $X$ over an algebraically
closed field $\bk$ of characteristic $p>0$. From \cite{Har98, HW02}, we know that klt singularities are closely related to strongly $F$-regular singularities
in positive characteristic. Moreover,
log canonical thresholds ($\lct$) correspond to $F$-pure thresholds
($\fpt$) in positive characteristic (see \cite{HW02}).
In spirit of Theorem \ref{thm:liueq}, we define the $F$-volume
of singularities in characteristic $p$ as follows.

\begin{defn}[\cite{Liu18b}]
 Let $X$ be an $n$-dimensional strongly $F$-regular variety over an algebraically
 closed field $\bk$ of positive characteristic. Let $x\in X$
 be a closed point. We define the \emph{$F$-volume}
 of $(x\in X)$ as
 \[
  \Fvol(x,X):=\inf_{\fa\colon\fm_x\textrm{-primary}}
  \fpt(X;\fa)^n\mult(\fa).
 \]
\end{defn}

Similar to \cite{dFEM}, Takagi and Watanabe \cite{TW04}
showed that if $x\in X$ is a smooth point, then $\Fvol(x,X)=n^n$.

Another interesting invariant of a strongly $F$-regular
singularity $x\in X$ is its $F$-signature
$s(x,X)$, see \cite{SVdB97, HL02, Tuc12}. 
In \cite{Liu18b}, we estabilish the following comparison result between
the $F$-volume and the $F$-signature.

\begin{thm}[\cite{Liu18b}]\label{t-com}
 Let $x\in X$ be an $n$-dimensional strongly $F$-regular singularity.
 Then
 \[
  n!\cdot s(x,X)\leq \Fvol(x,X)\leq n^n\min\{1, n!\cdot s(x,X)\}.
 \]
\end{thm}




It would be interesting to study the limiting behavior
of $F$-volumes of mod-$p$ reductions
of a klt singularity over characteristic zero when $p$ goes
to infinity.
\begin{conj} Let $x\in (X,\Delta)$ be a klt singularity over characteristic $0$. Let $x_p\in (X_p,\Delta_p)$ be its reduction mod $p\gg 0$,
then 
$$\hvol(x, X,\Delta)=\lim_{p\to \infty}\Fvol(x_p, X_p,\Delta_p).$$ 
\end{conj}

\begin{rem}Together with Theorem \ref{t-com}, this will imply that for the reductions $(X_p,\Delta_p)$, the F-signature $s(x_p,X_p,\Delta_p)$ has a uniform lower bound as $p\to \infty$, as asked in \cite[Question 5.9]{CRST16}.
\end{rem}

\subsubsection{Relation to local orbifold Euler numbers}

In \cite{Lan03}, Langer introduced local orbifold Euler numbers for general log canonical surface singularities and used it to prove a Miyaoka-Yau inequality for any log canonical surface. In an attempt to understand Langer's inequality using the K\"{a}hler-Einstein metric on a log canonical surface, 
Borbon-Spotti conjectured recently in \cite{BS17} that the volume densities of the singular K\"{a}hler-Einstein metrics should match Langer's local Euler numbers (at least for log terminal surface singularities). They verified this in special examples by comparing the known values of both sides. 
On the other hand, from Theorem \ref{thm-hvol2Theta}, we know that the normalized volume is equal to the volume density up to a factor $(\dim X)^{\dim X}$ for any point $(X,x)$ that lives on a Gromov-Hausdorff limit of smooth K\"{a}hler-Einstein manifolds (\cite{HS17, LX17}). In view of this connection, one can formulate a purely algebraic problem about two algebraic invariants of the singularities.
This problem was already posed by in \cite{BS17} at least in the log terminal case. We formulate the following form by including one of Langer's expectations (see \cite[p.381]{Lan03}):
\begin{conj}[{see \cite[p.37]{BS17}}]\label{conj}
Let $(X, D, x)$ be a germ of log canonical surface singularity with $\bQ$-boundary. Then we have
\begin{equation}
e_{\rm orb}(x, X, D)=
\left\{
\begin{array}{ll}
\frac{1}{4}\hvol(x, X,D), & \text{ if } (X,D) \text{ is log terminal };\\
0, & \text{ if } (X, D) \text{ is not log terminal}.
\end{array}
\right. 
\end{equation}
\end{conj}

In \cite{Li18}, it was proved that the above conjecture is true when $(X, D, x)$ is a $2$-dimensional log-Fano cone or a log-CY cone. In particular, combined with Langer's calculation, one gets the local orbifold Euler numbers of line arrangements.
\begin{prop}[{\cite{Lan03, Li18}}]
Let $L_1, \dots, L_n$ be $m$ distinct lines in $\bC^2$ passing through $0$. Let $D=\sum_{i=1}^m \delta_i L_i$, 
where $0\le \delta_1\le \delta_2\le \dots \le \delta_m\le 1$. Denote $\delta=\sum_{i=1}^m \delta_i$. Then we have:
\begin{equation}
e_{\rm orb}(0, \bC^2, D)=\left\{
\begin{array}{lcc}
0 & \text{ if} &  (\bC^2, D, 0) \text{ is not klt } ;\\
(1-\delta+\delta_m)(1-\delta_m)& \text{ if } & \delta<2\delta_m;\\
 \frac{(2-\delta)^2}{4} & \text{ if } & 2\delta_m\le \delta\le 2.
\end{array}
\right.
\end{equation}
 \end{prop}
Here we point out a possible application of Theorem \ref{thm-2dim} (i.e. 2-dimensional case conjecture \ref{conj-local}) for studying Conjecture \ref{conj} for any log terminal singularity $(x, X, D)$. First, by Theorem \ref{thm-2dim} there exists a unique Koll\'{a}r component $S\cong \bP^1$ which minimizes the normalized volume. Let $\mu: Y\rightarrow X$ be the extraction of $S$ and $\Delta={\rm Diff}_S(D)$. By Theorem \ref{t-high} we know that $(S, \Delta)\cong (\bP^1, \sum_i \delta_i p_i)$ is indeed K-semistable (see \cite[section 6]{LX16}). Then $\mathscr{F}:=\Omega^1(\log (S+D))$ (defined using ramified coverings as in \cite{Lan03}) restricted to $S$ fits into an exact sequence of orbifold sheaves:
\begin{equation}
0\longrightarrow \Omega^1_S(\log(\Delta))\rightarrow \mathscr{F}|_S\rightarrow \cO_S\rightarrow 0.
\end{equation}
By \cite[Theorem 1.3]{Li18}, we know that $\mathscr{E}:=\mathscr{F}|_S$ is slope semistable. Then the generalization of \cite[Proposition 3.16]{Wah93} to the logarithmic/orbifold setting combined together with Langer's work should imply that $e_{\rm orb}(x, X, D)=\frac{c_1(\mathscr{E})^2}{4(-S\cdot S)_Y}$ which is indeed equal to $\frac{\hvol(\ord_S)}{4}$. 

\subsubsection{Normalized volume function}
We have mainly concentrated on the minimizer of the normalized volume function. We can also ask questions on the general behavior of the normalized volume function. For example:
\begin{ques}[Convexity] Let $\sigma\subset \Val_{X,x}$ be a simplex of quasi-monomial valuations. Is it true that $\hvol(\cdot)$ is always convex on $\sigma$? Is there a more general convexity property for $\hvol$ on $\Val_{X,x}$?
\end{ques}

\begin{ques}
Is the normalized volume a lower semicontinuous function on $\Val_{X,x}$? If this is true, then it would directly imply the existence of minimizer of $\hvol$ using the properness estimate in Theorem \ref{t-izumi}.
\end{ques}






\end{document}